\documentclass[american]{amsart}
\usepackage{amsmath, amsfonts, amssymb, amscd, amsthm, amsbsy, esint, bbm, epsf, mathtools}
\usepackage[utf8]{inputenc}

\usepackage{datetime}
\usepackage[margin=1in]{geometry}
\usepackage{comment}

\mathtoolsset{showonlyrefs}

\usepackage[colorlinks,linkcolor = blue,citecolor = green]{hyperref}

\newtheorem{theorem}{Theorem}[section]
\newtheorem{proposition}{Proposition}[section]
\newtheorem{corollary}{Corollary}
\newtheorem{lemma}{Lemma}[section]
\theoremstyle{definition}
\newtheorem{definition}{Definition}
\newtheorem{remark}{Remark}

\numberwithin{equation}{section}

\DeclareMathOperator{\R}{\mathbb{R}}

\DeclareMathOperator{\M}{\mathcal{M}}

\let\T\relax

\DeclareMathOperator{\T}{\widetilde T}

\DeclareMathOperator{\embeds}{\hookrightarrow}
\DeclareMathOperator{\weakstar}{\overset{\ast}{\rightharpoonup}}
\DeclareMathOperator*{\esssup}{ess\,sup}
\DeclareMathOperator*{\essinf}{ess\,inf}
\DeclareMathOperator*{\supp}{supp}
\DeclareMathOperator{\eps}{\varepsilon}

\newcommand{\set}[1]{\left\{ #1 \right\}}

\title[Regularity for isentropic Euler]{Unconditional regularity and trace results for the isentropic Euler equations with $\gamma = 3$} 
\author[Golding]{William Golding}
\address[William Golding]{\newline Department of Mathematics, \newline The University of Texas at Austin, Austin, TX 78712, USA}
\email{wgolding@utexas.edu}

\date{\today}

\thanks{\textbf{Acknowledgment.} 
The author would like to thank his advisor, Alexis Vasseur, for suggesting this problem, for fruitful and insightful conversations, for continued encouragement, and for being a superb mentor. W. Golding is partially supported by NSF grant: DMS1840314.
}

\begin{document}

\begin{abstract}
In this paper, we study the regularity properties of bounded entropy solutions to the isentropic Euler equations with $\gamma = 3$. First, we use a blow-up technique to obtain a new trace theorem for all such solutions. Second, we use a modified De Giorgi type iteration on the kinetic formulation to show a new partial regularity result on the Riemann invariants. We are able to conclude that in fact for any bounded entropy solution $u$, the density $\rho$ is almost everywhere upper semicontinuous away from vacuum. To our knowledge, this is the first example of a nonlinear hyperbolic system, which fails to be Temple class, but has the property that generic $L^\infty$ initial data give rise to bounded entropy solutions with a form of near classical regularity. This provides one example that $2\times 2$ hyperbolic systems can possess some of the more striking regularizing effects known to hold generically in the genuinely nonlinear, multidimensional scalar setting. While we are not able to use our regularity results to show unconditional uniqueness, the results substantially lower the likelihood that current methods of convex integration can be used in this setting.

\end{abstract}

\maketitle

\tableofcontents

\section{Introduction}\label{sec:intro}

In this paper, we study the properties of bounded entropy solutions $u = (\rho, m)$ to the isentropic Euler equations with adiabatic index $\gamma = 3$: 
\begin{equation}\label{eqn:Euler}
\begin{cases}
\rho_t + m_x = 0\\
m_t + \left(\frac{m^2}{\rho} + \frac{\rho^3}{12}\right)_x = 0,
\end{cases}
\end{equation}
where $(t,x)\in \R^+ \times \R$ are time and space, respectively, and $(\rho,m)$ are the unknown mass and momentum densities. We consider solutions $(\rho,m)$ which satisfy \eqref{eqn:Euler} in the sense of distributions and are entropic for \emph{all} entropy, entropy-flux pairs. More precisely, we require that for any pair $(\eta, q)$, such that $\eta$ is convex, $\eta,\ q \in C(\R^+\times \R)$, and $q^\prime = \eta^\prime f^\prime$, where $f(\rho,m) = (m, m^2/\rho + \rho^3/12)^t$ is the flux associated to \eqref{eqn:Euler}, $u$ satisfies the entropy inequality,
\begin{equation}\label{eqn:entropy}
\eta(u)_t + q(u)_x \le 0,
\end{equation}
in the sense of distributions. Throughout, we will work with solutions which are globally bounded, in the sense that $\rho, \frac{m}{\rho}\in L^\infty(\R^+ \times \R)$, so that the physical quantities are bounded. We recall that the global-in-time existence of such solutions to \eqref{eqn:Euler} for general initial datum $u_0\in L^\infty$ is known provided $\rho_0,\frac{m_0}{\rho_0}\in L^\infty$ and $\rho_0 \ge 0$ (see \cite[Theorem 17.8.1]{dafermos_book}). We comment that for \eqref{eqn:Euler} existence follows from the now classical method of compensated compactness introduced by Tartar and Murat. We refer the reader to \cite[Section 17.9]{dafermos_book} for a detailed history of the existence theory for $2\times 2$ systems.

\medskip

However, for more regular initial data $u_0$, there are several methods of constructing solutions, of which we list three:
\begin{itemize}
\item Glimm's scheme (Random choice)
\item The front tracking method
\item The method of vanishing viscosity
\end{itemize}
Each of these three methods of constructing solutions is known to give rise to the same entropy solution $u$, which we will refer to as Glimm's solution, for a quite general class of fluxes, provided $\|u_0\|_{BV}$ is sufficiently small. 
Yet, even for $\|u_0\|_{BV} < \eps$, it is open whether the Glimm entropy solution is the only one obtained by compensated compactness methods. We remark that for Temple class systems, which are in some sense degenerate, this is more easily resolved (see \cite{dafermos_geng, bressan_goatin2, heibig} and references therein). However, to our knowledge, this problem remains wide open for every non-Temple class $2\times 2$ system. This is highly unsatisfactory because, at least heuristically, fully discretized numerical schemes (i.e. widely used finite volume and finite difference based methods) are naturally represented as measure-valued solutions and converge to compensated compactness solutions. For a more complete discussion of these results, we direct the reader to the celebrated paper of Bianchini and Bressan \cite{bianchini_bressan} and the references therein.

\medskip

This motivates our study of solutions satisfying only the a priori bounds $\rho, \frac{m}{\rho} \in L^\infty(\R^+ \times \R)$, which are the strongest a priori bounds naturally guaranteed for reasonable compensated compactness schemes via the existence of invariant regions. 
We emphasize again that none of the results below make any additional regularity assumptions such as $BV_{loc}$ nor any smallness assumptions and are in this sense truly unconditional. 
As a first step towards resolving the aforementioned conjecture for the specific system \eqref{eqn:Euler}, we study two closely related regularity properties for bounded solutions, namely, trace properties and local $BV$-like structure.

\subsection{Trace Properties}

In the case of scalar conservation laws with a convex flux, it is known the Kru\v{z}kov entropy solution regularizes from $L^\infty$ into $BV$ \cite{kruzkov,oleinik}. However, for multidimensional scalar conservation laws, it is well known that solutions may not be $BV$, but still display striking regularization properties (see \cite{lions_perthame_tadmor1,panov,kwon_vasseur,vasseur_trace,chen_frid,chen_rascle,delellis_otto_westdickenberg,silvestre} and references therein). One of the earliest examples of this phenomenon is the fractional Sobolev regularization implied by the kinetic formulation introduced by Lions, Perthame, and Tadmor in \cite{lions_perthame_tadmor1}. Thereafter, the kinetic formulation was used by Vasseur in \cite{vasseur_trace} to show a one-sided trace property for solutions, which was later strengthened in \cite{panov, kwon_vasseur}. In particular, this implies that solutions have well-defined boundary values along lower dimensional hyper-surfaces, and these boundary values are obtained in a strong topology. 

\medskip

On the other hand, for $2\times 2$ systems, not only is regularization unknown for a generic, non-Temple-class system, but we lack even specific examples of systems known to possess regularizing effects, excepting the fractional Sobolev regularization for \eqref{eqn:Euler} implied by the kinetic formulation introduced by Lions, Perthame, and Tadmor in \cite{lions_perthame_tadmor2}. We are aware of exactly one trace theorem: in \cite{vasseur_kinetic}, Vasseur shows any bounded solution $u$ to \eqref{eqn:Euler} satisfies $u\in C(0,T;L^1_{loc})$, which we interpret as a two-sided trace property in time. In light of the increasingly evident role spatial traces play in the uniqueness and stability theory for systems (see the discussion of uniform traces below), the first main goal of this paper is to extend the temporal traces considered by Vasseur in \cite{vasseur_kinetic} to $L^1$-type spatial traces as done in the scalar case in \cite{vasseur_trace}. 

\medskip

In order to state our result, we introduce the following particular notion of a one-sided trace, used frequently below.
\begin{definition}[Strong Trace Property]\label{defn:strong_traces}
We say a function $u:\R^+ \times \R \rightarrow \R^n$, $u\in L^\infty(\R^+\times \R)$ satisfies the strong trace property if for any $h: [0,T] \rightarrow \R$ Lipschitz, there are trace functions $u^+,u^-\in L^\infty([0,T];\R^n)$ that are obtained as
\begin{equation}
\lim_{k\rightarrow \infty} \esssup_{y \in (0,1/k)} \int_0^T |u(t,h(t) + y) - u^+(t)| + |u(t,h(t) - y) - u^-(t)| \ dt = 0. 
\end{equation}
\end{definition}

\begin{remark}\label{rmk:trace2}
Of course, as $u,\ u^+,\ \text{and}\ u^-$ are all bounded, one may interpolate with the $L^\infty$ bound and put a power $p$ inside the trace condition as in Definition \ref{defn:strong_traces} for any $1 \le p < \infty$.
\end{remark}

With this definition in hand, we are ready to state our first main result:
\begin{theorem}\label{thm:trace}
Suppose $u:\R^+ \times \R \rightarrow \R^2$ is a bounded entropy solution to \eqref{eqn:Euler}. Then, $u$ has strong traces in the sense of Definition \ref{defn:strong_traces}.
\end{theorem}

\begin{remark}\label{rmk:trace1}
Since strong traces are obtained strongly in an $L^1$ sense, nonlinear functions of a solution $u$ to \eqref{eqn:Euler} also have strong traces. For example, we have enough regularity to make sense of the characteristic fields $\lambda_i(u^\pm)$ as well the Rankine-Hugoniot condition.
\end{remark}

Combining Theorem \ref{thm:trace} with a result of \cite{leger_vasseur}, we obtain the following corollary:
\begin{corollary}\label{cor:Dafermos}
Suppose $u:\R^+ \times \R \rightarrow \R^2$ is a bounded entropy solution to \eqref{eqn:Euler}. Then, for any $h: [0,T] \rightarrow \R$ a Lipschitz curve, for almost every $t\in [0,T]$, either $u^+(t) = u^-(t)$ or $(u^-(t), u^+(t))$ is an admissible shock with speed $\dot{h}(t)$.
\end{corollary}
This corollary is interesting because it translates the a priori regularity provided by Theorem \ref{thm:trace} into a tool that can be used unconditionally in the study of solutions to \eqref{eqn:Euler}. The proof of the corollary follows immediately from \cite[Lemma 6]{leger_vasseur}, which is proved in the Appendix for the reader's convenience.

\medskip

Let us introduce one more trace property, which we refer to as the uniform trace property to distinguish it from strong traces. The following property is never used below, but is introduced formally to provide motivation for our later results.
\begin{definition}[Uniform Trace Property]\label{defn:uniform_traces}
We say a function $u:\R^+ \times \R \rightarrow \R^n$, $u\in L^\infty(\R^+\times \R)$ satisfies the uniform trace property if for any $h: [0,T] \rightarrow \R$ Lipschitz, there are trace functions $u^+,u^-\in L^\infty([0,T];\R^n)$ that are obtained as
\begin{equation}
\lim_{k\rightarrow \infty} \int_0^T \esssup_{y \in (0,1/k)} |u(t,h(t) + y) - u^+(t)| + |u(t,h(t) - y) - u^-(t)| \ dt = 0. 
\end{equation}
\end{definition}
The uniform trace property has been present for some time in the theory of $a$-contractions with shifts originating in the work of Leger \cite{leger} (for scalar) and Leger and Vasseur \cite{leger_vasseur} (for systems). Recently, it has reappeared as the most general class under which Glimm's solutions remain unique. More precisely, \cite[Theorem 1.3]{chen_krupa_vasseur}\textemdash proved in \cite{chen_krupa_vasseur} and the companion paper \cite{golding_krupa_vasseur}\textemdash states that for any small $BV$ initial datum $u_0$, if $u$ and $v$ are bounded entropy solutions (to a class of generic $2\times 2$ systems) with the same initial datum $u_0$ and $u$ and $v$ satisfy the uniform trace property in the sense of Definition \ref{defn:uniform_traces}, then $u = v$. This theorem improves several preceding uniqueness results (\cite{bressan_lewicka,bressan_lefloche,bressan_goatin}), yet still only provides conditional uniqueness for small $BV$ initial datum, because we do not know whether the uniform trace property holds for general $L^\infty$ solutions. This leads to two natural questions which motivate the remainder of our results: First, does special structure of \eqref{eqn:Euler} let us prove uniform traces for entropy solutions? Second, can we make the uniqueness theorem work using only strong traces? We are only able to provide partial answers to both of these questions.

\subsection{Improved Traces}

We seek first to improve the regularity results of Theorem \ref{thm:trace} and bridge the gap between strong and uniform traces in the sense of Definitions \ref{defn:strong_traces} and \ref{defn:uniform_traces}. To state our main result in this direction, let us recall that for the $\gamma = 3$ isentropic Euler equations as written in \eqref{eqn:Euler}, the Riemann invariants are the characteristic fields, which are given by
\begin{equation}
\lambda_1(\rho,m) = \frac{m}{\rho} - \frac{\rho}{2} \qquad \text{and}\qquad \lambda_2(\rho,m) = \frac{m}{\rho} + \frac{\rho}{2}.
\end{equation}
Then, we are able to obtain the following improvement of the trace condition for $\rho$, $\lambda_1$, and $\lambda_2$.
\begin{theorem}\label{thm:improved_traces}
Suppose $u:\R^+ \times \R \rightarrow \R^2$ is a bounded entropy solution to \eqref{eqn:Euler} and further suppose $\essinf \rho > 0$. Then, for any Lipschitz curve $h:[0,T]\rightarrow \R$, for each of $g_1 = \lambda_2, \ \text{or} \ -\lambda_1$,
\begin{equation}\label{eqn:improved_traces1}
\lim_{n\rightarrow \infty} \int_0^T \esssup_{-\frac{1}{n} < y < \frac{1}{n}} \left[g_1(t,h(t) + y) - \max(g_1^+(t),g_1^-(t))\right]_+ \;dt = 0.
\end{equation}
Moreover, for any sequence $r_n\rightarrow 0^+$ and any $0 < \eps < 1$, for $g_2 = \rho,\ \lambda_2, \ \text{or} \ -\lambda_1$,
\begin{equation}\label{eqn:improved_traces2}
\lim_{n\rightarrow \infty} \int_0^T \esssup_{\eps r_n < y < r_n} \left(\left[g_2(t,h(t) + y) - g_2^+(t)\right]_+ + \left[g_2(t,h(t) - y) - g_2^-(t)\right]_+\right) \;dt = 0.
\end{equation}
\end{theorem}

\begin{remark}
We can easily recover \eqref{eqn:improved_traces1} for $\rho$ with $\max(g^+_1,g^-_1)$ replaced by $\max(\lambda_2^+,\lambda_2^-) - \min(\lambda_1^+,\lambda_1^-)$ using $\rho = \lambda_2 - \lambda_1$. We note that one can show using Corollary \ref{cor:Dafermos} and explicit computations along the (globally defined) shock curves of \eqref{eqn:Euler} that $\lambda_i^- = \max(\lambda_i^-,\lambda_i^+)$ almost everywhere along any Lipschitz curve $h$. However, this implies we cannot recover \eqref{eqn:improved_traces1} for $\rho$ with $\max(\rho^-,\rho^+)$ using our method, since $\max(\rho^-,\rho^+) < \lambda_2^- - \lambda_1^+$ may occur.
\end{remark}

Theorem \ref{thm:improved_traces} provides a significant improvement over Theorem \ref{thm:trace} in that we are able to pass the supremum in $x$ inside the integral and obtain traces that are quite close to the uniform traces of Definition \ref{defn:uniform_traces}. There appear to be three deficiencies with this method: First, we lose traces on the momentum $m$. Second, we only control the oscillation of $\rho$ and the Riemann invariants in one direction, either up or down. Third, we lose control of the oscillations near the curve, i.e. over the region $0 < y < \eps r_n$.

\medskip

In some ways, the first and second problem are actually one and the same. We only directly control the oscillation of $\lambda_1$ and $\lambda_2$ (in opposite directions) and recover control of $\rho$ via the algebraic relation $\rho = \lambda_2 - \lambda_1$. Thus, it seems unlikely that we can use the relation $m = \frac{\lambda_2^2 - \lambda_1^2}{2}$ to obtain control of $m$. We will explore this difficulty more below. 
By contrast, the third problem is not so bad. Indeed, for most practical applications traces of the form \eqref{eqn:improved_traces1}, which hold all the way up to the curve, are sufficient.

\medskip

Although we are not able to use Theorem \ref{thm:improved_traces} to obtain any significant unconditional uniqueness result for small $BV$ solutions or even small shocks to \eqref{eqn:Euler}, Theorem \ref{thm:improved_traces} still significantly restricts the  possibility of convex integration of \eqref{eqn:Euler}. In the $L^\infty$ framework described here, considering a $1d$ shock to \eqref{eqn:Euler} as a $2d$ planar shock and convex integrating in $2d$ yields solutions which fail to have traces in any strong topology (see \cite{chiodaroli_delellis_kreml} for more details). The method of convex integration typically uses large oscillations to create low regularity weak solutions which exhibit wild, unexpected behavior. Theorems \ref{thm:trace} and \ref{thm:improved_traces} can be interpreted as restricting the type of high frequency oscillations one can use in building a convex integration scheme so that one still obtains an entropy solution at the end of the process. That is, while Theorem \ref{thm:improved_traces} does not close the uniqueness gap, it provides further evidence to support the conjecture that (at least for small $BV$ initial data) there is only one bounded entropy solution to \eqref{eqn:Euler}.

\subsection{Semicontinuity of Solutions}

Our final result concerns a form of classical regularity of solutions to \eqref{eqn:Euler}. Due to the presence of genuine nonlinearity, solutions are expected to be generically piecewise smooth, with jumps occurring only at entropic shock discontinuities.
In the context of multidimensional scalar equations, where solutions need not regularize to $BV$, there have been several attempts to characterize the regularizing effects of genuine nonlinearity by showing $BV$-like properties for solutions, which is a convenient mathematical relaxation of the expected regularity. For instance, in \cite{delellis_otto_westdickenberg} the authors characterize blow-up limits and prove for a given entropy solution, there is an $\mathcal{H}^{n-1}$-rectifiable jump set $\mathcal{J}$ outside of which the solution is $VMO$. Further improvements in this direction have been made (see for example \cite{bianchini_marconi1, bianchini_marconi2} and references therein), with a particularly notable recent contribution by Silvestre.
In \cite{silvestre}, Silvestre uses a De Giorgi iteration to show that outside of $\mathcal{J}$, the solution is continuous in the sense that each point $x\notin \mathcal{J}$ is a Lebesgue point at which all blow-ups converge in $L^\infty_{loc}$.

Our last result is obtained via a modification of Silvestre's method and constitutes a similar, but weaker, result for the system \eqref{eqn:Euler}.
To state the result formally, 
we introduce some appropriate quantities which are defined everywhere, independently of the chosen version.
Namely, for $g:\R\times \R \rightarrow \R$, we define the following upper (resp. lower) semicontinuous envelopes of $g$ via
\begin{equation}
\overline{g}(t,x) \coloneqq \lim_{r \rightarrow 0^+}\esssup_{(\tau,y)\in B_r(t,x)} g(\tau, y) \qquad \text{and} \qquad \underline{g}(t,x) \coloneqq \lim \essinf_{(\tau,y)\in B_r(t,x)} g(\tau,y).
\end{equation}
Note that for $g\in L^\infty_{loc}$, $\overline{g}$ is always upper semicontinuous while $\underline{g}$ is always lower semicontinuous.
We our now ready to state our final result:
\begin{theorem}\label{thm:cont}
Suppose $u:\R^+ \times \R \rightarrow \R^2$ is a bounded entropy solution to \eqref{eqn:Euler}. Then,
\begin{equation}
\rho = \overline{\rho}, \quad \lambda_1 = \underline{\lambda}_1, \quad \text{and} \quad \lambda_2 = \overline{\lambda}_2 \text{ almost everywhere on }\{\underline{\rho}\neq0\}
\end{equation}
and
\begin{equation}
m = \begin{cases}
	&\overline{m} \text{ almost everywhere on } \{\underline{\rho}\neq 0\} \cap \{\underline{\lambda}_1 \ge 0 \}\\
	&\underline{m} \text{ almost everywhere on } \{\underline{\rho}\neq0\} \cap \{ \overline{\lambda}_2 \le 0 \}
	\end{cases}.
\end{equation}
\end{theorem}

Fundamentally, Theorem \ref{thm:cont} is a rigorous formulation of the claim that for a non-vacuum solution, the density, and the Riemann invariants are each semicontinuous almost everywhere. Moreover, for a large number of solutions we can even regain semicontinuity of the momentum, $m$. Note that we obtain a form of classical regularity for the conserved quantities. To our knowledge, this is the only example of a non-Temple-class system where we can show the production of any classical regularity. In particular, while our methods apply strictly to the $\gamma = 3$ Euler system, it is unclear whether this production of semicontinuity is generic for $2\times 2$ systems, specific to the $\gamma = 3$ Euler system, or specific to systems with sufficiently many entropies.

\begin{remark}
Note, Theorem \ref{thm:cont} only works away from vacuum. This stands in contrast to the scalar case considered by Silvestre in \cite{silvestre}, where there is no problem with vacuum. Our problem with vacuum seems to stem from the limitations of the kinetic formulation of \eqref{eqn:Euler} which is central to the proof of both Theorem \ref{thm:trace} and \ref{thm:cont}. The kinetic formulation loses meaning at vacuum states, since it is derived from the so-called weak entropies, which vanish at vacuum. Therefore, for example, the kinetic formulation has trouble differentiating two vacuum states where nearby, the velocity $\frac{m}{\rho}$ is drastically different. To circumvent this problem, one might attempt to use the relative entropy method developed by Dafermos and DiPerna in \cite{dafermos, diperna} which often is able to overcome the difficulties posed by vacuum. However, merging the two methods seems difficult given the low level of regularity present.
\end{remark}

\begin{remark}
In Theorem \ref{thm:cont}, we have no explicit description of the set of discontinuity points of an entropy solution $u$. Instead, we rely upon Lebesgue's Differentiation Theorem to say that almost every $(t,x)$ is a Lebesgue point. In \cite{delellis_otto_westdickenberg}, the jump set $\mathcal{J}$ is explicitly defined in terms of the entropy dissipation measure and it is shown that each $(t,x)\notin \mathcal{J}$ is a point of $VMO$, using a Liouville theorem (\cite[Proposition 6]{delellis_otto_westdickenberg}) for blow-ups at such points. We expect one can show a similar result here and also the analogous statement that $\mathcal{J}$ is codimension $1$ rectifiable. However, such a result relies upon on the notion of blow-ups at single points, which are surprisingly different from the notion of blow-ups along Lipschitz curves considered here. Since such an analysis would take us too far afield, these problems are left for future work.
\end{remark}

\subsection{Proof Overview}

In Section \ref{sec:preliminaries}, we begin by introducing the kinetic formulation of \eqref{eqn:Euler} and the additional tools we gain from it in the form of averaging lemmas, which originate in the work of Golse, Lions, Perthame, and Sentis in \cite{golse_lions_perthame_sentis} with roots in the earlier works of Golse, Perthame, and Sentis and, separately, Agoshkov in \cite{golse_perthame_sentis,agoshkov}. We also introduce a few notational conventions used throughout.

\medskip

In Section \ref{sec:trace}, we use the kinetic formulation to prove Theorem \ref{thm:trace}. The proof exploits the same particular structure of \eqref{eqn:Euler} and employs similar techniques as the proof of Vasseur's result that $u\in C(0,T;L^p_{loc})$. Similar to \cite{vasseur_trace}, there are two new challenges in the proof. Vasseur's theorem implies two-sided traces along flat space-like curves. By contrast, Theorem \ref{thm:trace} implies one-sided traces along Lipschitz time-like curves. As neither \eqref{eqn:Euler} nor its kinetic formulation are symmetric in $t$ and $x$, we cannot simply interchange their roles. In addition, some care has to be taken when straightening the time-like curves to ensure one has sufficient regularity to perform all the desired arguments and that the additional terms associated to the geometry of the curve do not affect the proof.

\medskip

In Section \ref{sec:DeGiorgi}, we use the kinetic formulation to prove a new partial regularity result for the Riemann invariants $\lambda_1(u)$ and $\lambda_2(u)$:
\begin{proposition}\label{prop:Linfty}
Suppose $u = (\rho,m)$ is an entropy solution to \eqref{eqn:Euler} with $\|\rho\|_{L^\infty(B_2)} + \|m/\rho\|_{L^\infty(B_2)} \le \Gamma$ and $\essinf_{(t,x) \in B_2} \rho(t,x) \ge M > 0$. Suppose $\overline{u} = (\overline{\rho},\overline{m})$ is a fixed non-vacuum state. Then, there is an $\eps_0 = \eps_0(\Gamma, M) > 0$, $\alpha =\frac{1}{21}$, and $\widetilde{C} = \widetilde{C}(\Gamma,M)$ such that for any $0 < \eps < \eps_0$,
\begin{align}
\int_{B_2}\left(\lambda_1(\overline{u}) - \lambda_1(u)\right)_+ \;dxdt < \eps \qquad &\text{ implies } \qquad \esssup_{(t,x)\in B_1} \left(\lambda_1(\overline{u}) - \lambda_1(u(t,x))\right)_+ < \widetilde{C}\eps^\alpha\\
\int_{B_2}\left(\lambda_2(u) - \lambda_2(\overline{u})\right)_+ \;dxdt < \eps \qquad &\text{ implies } \qquad \esssup_{(t,x)\in B_1} \left(\lambda_2(u(t,x)) - \lambda_2(\overline{u})\right)_+ < \widetilde{C}\eps^\alpha.
\end{align}
\end{proposition}

This result implies that if the $L^1$-oscillation below (resp. above) a fixed threshold is sufficiently small in a ball, then the $L^\infty$-oscillation below (resp. above) the threshold is controlled quantitatively on a smaller ball. The above result is the main new ingredient in the proof of Theorem \ref{thm:cont} and involves a new application of the first lemma of De Giorgi to conservation laws. Originally introduced by De Giorgi in \cite{degiorgi} for the study of elliptic equations, recently the method has seen novel applications in the areas of kinetic theory \cite{silvestre_imbert,mouhot_guerand,golse_imbert_mouhot_vasseur} and elsewhere (for example, see \cite{stokols_vasseur,caffarelli_vasseur} and the references therein), and most relevantly for us, in the work of Silvestre on scalar conservation laws \cite{silvestre}. In the kinetic framework studied by Silvestre in the scalar case, the lack of a dissipative right hand side means that hypoelliptic estimates are unavailable. Instead, one must rely upon the dispersive nature of the free transport operator to gain integrability, reminiscent of the ideas of \cite{golse_imbert_mouhot_vasseur}. Unlike in the scalar case, we are unable to work directly on the entropy solution $u$, but instead are forced to work on the Riemann invariants. Even then, we are still only able to obtain one-sided results. Yet, because of the algebraic identities,
\begin{equation}
\rho = \lambda_2(u) - \lambda_1(u) \qquad \text{ and } \qquad m = \frac{\left[\lambda_2(u)\right]^2 - \left[\lambda_1(u)\right]^2}{2},
\end{equation}
we are still able to leverage the one-sided results to gain information about the original conserved quantities. In Section \ref{sec:DeGiorgi}, we attempt to highlight precisely when and why these differences with the scalar case arise.

\medskip 

In Section \ref{sec:semicontinuity}, we combine Proposition \ref{prop:Linfty} with Lebesgue's Differentiation Theorem to prove Theorem \ref{thm:cont}. Actually, Theorem \ref{thm:cont} will be derived from a slightly more precise result, concerning the behavior of $u$ at any point of vanishing mean oscillation.

\medskip

In Section \ref{sec:ODE}, we combine Proposition \ref{prop:Linfty} with the $L^1$ control provided by Theorem \ref{thm:trace} to prove Theorem \ref{thm:improved_traces}.

\section{Preliminaries}\label{sec:preliminaries}

\subsection{Kinetic Formulation and Averaging Lemmas}

Here we begin by recalling that the kinetic formulation of scalar conservation laws introduced in \cite{lions_perthame_tadmor1} has an analogue for the isentropic Euler system. For $\gamma = 3$, the associated kinetic formulation is purely kinetic and takes the following particularly simple form:
\begin{equation}\label{eqn:kinetic}
\left\{\begin{array}{ll}
\partial_t f + v \partial_x f = -\partial_{vv}\mu\\[3pt]
f(t,x,v) = \chi_{[a(t,x),b(t,x)]}(v) \text{ for almost every }(t,x)\\[3pt]
\supp(f(t,x)) \subset [-L,L],
\end{array}\right.
\end{equation}
where $f(t,x,v)$ is our unknown, which has been augmented with the velocity variable $v\in \R$, and $\mu$ is a finite (non-negative) measure. From the kinetic equation, we recover \eqref{eqn:Euler} via the relations
\begin{equation}\label{eqn:moments}
\rho(t,x) = \int_{-L}^L f(t,x,v) \;dv \qquad \text{and}\qquad m(t,x) = \int_{-L}^L vf(t,x,v) \;dv,
\end{equation}
which imply, in fact, $a(t,x) = \lambda_1(t,x)$ and $b(t,x) = \lambda_2(t,x)$.
In this form, the following is known:
\begin{theorem}[\protect{\cite[Theorem 3]{lions_perthame_tadmor2}}]\label{thm:kinetic}
Let $u = (\rho,m)$ be a bounded entropy solution to \eqref{eqn:Euler}. Then, there is a unique pair $(f, \mu)$ such that $f\in L^\infty$, $\mu$ is a finite Borel measure, $f$ and $\mu$ solve \eqref{eqn:kinetic}, and $f$ is related to $(\rho,m)$ via \eqref{eqn:moments}.
\end{theorem}

Next, we have the following averaging lemma, from which we will gain compactness of sequences of solutions to \eqref{eqn:kinetic} in Section \ref{sec:trace}:
\begin{lemma}[\protect{\cite[Theorem B]{lions_perthame_tadmor1}}]\label{lem:averaging1}
Let $1 < p \le 2$ and say $f_k$ is a sequence of distributional solutions to the equation
\begin{equation}
\partial_t f_k + a(v)\partial_x f_k = (1- \Delta_{x,t})^{1/2}(1 - \partial_{vv})^{r/2} g_k,
\end{equation}
for some function $a\in C^\infty(\R)$, some $r > 0$, and some sequence $g_k$. If $\{g_k\}$ is compact in $L^p([0,T]\times \R \times \R)$, $\{f_k\}$ is bounded in $L^p_{loc}([0,T]\times \R \times \R)$, and $a$ satisfies
\begin{equation}\label{eqn:nondegenerate}
\left|\set{(t,x) \big| t^2 + x^2a(v) = 0}\right| = 0,
\end{equation}
then for any $\psi \in L^q(\R)$ for $q = \frac{p}{p-1}$, $\{\int_{\R} f_k(\cdot,\cdot,v) \psi(v) \ dv\}$ is compact in $L^p_{loc}([0,T]\times \R)$.
\end{lemma}

Finally, we will also need quantitative averaging lemmas, which will be used in Section \ref{sec:DeGiorgi} to gain integrability. For a comprehensive study of averaging lemmas, we refer to the treatise of DiPerna, Lions, and Meyer \cite{diperna_lions_meyer}, however, the specific one we will use is:
\begin{lemma}[\protect{\cite[Averaging Lemma 2.1]{tadmor_tao}}]\label{lem:averaging2}
Suppose $f \in L^2$ is a solution to
\begin{equation}
\partial_t f + v\partial_x f = \partial_{vv} \mu_2 + \partial_{v}\mu_1 + \mu_0 + g,
\end{equation}
where $\mu_0,\mu_1,\mu_2 \in \M$, $g\in L^1$. Then, for $\theta\in (0,1/7)$, $r = \frac{7}{4}$, and for any $\psi \in C_c^\infty([-2,2])$ a bump function,
\begin{equation}
\left\|\int f(t,x,v)\psi(v) \ dv \right\|_{W^{\theta,r}_{t,x}} \lesssim \|f\|_{L^2}^{6/7}\left(\|g\|_{L^1} + \|\mu_0\|_{TV} + \|\mu_1\|_{TV} + \|\mu_2\|_{TV}\right)^{1/7}.
\end{equation}
\end{lemma}

\subsection{Notation}

\begin{itemize}
\item In the following, $L^p(D)$ denotes the standard Lebesgue space for $1\le p \le \infty$ on the domain $D$. When the domain of integration is clear from context, we will frequently write $L^p$.

\item Inessential constants that may be replaced by any larger number and may change from line to line are denoted by $C$. We will also make use of the convention that $a \lesssim_{\Lambda} b$ means there exists a constant $C = C(\Lambda)$ such that $a \le Cb$.

\item Since we deal with blow-up solutions, frequently, we deal simultaneously with two coordinate systems. We will attempt to stick to the following naming convention: $h$ will denote a fixed Lipschitz curve; $(t,x) = (t,h(t)) \in [0,T]\times \R$ will denote a location in physical space we are zooming in about; and, $(\tau, y) \in \R^2$ will denote the local coordinates of the blow-up.

\item Finally, we will frequently use several notations for space-time balls: for us $B_r(x) \subset \R^2$ is the ball centered at $x\in \R^2$ of radius $r$. If no reference is made to the center, $x = 0$. Note all balls should be considered to be a subset of $\R^2$ unless otherwise stated.
\end{itemize}

\section{Existence of Strong Traces}\label{sec:trace}

In this section, we prove the following proposition, which immediately implies Theorem \ref{thm:trace}.
\begin{proposition}\label{prop:trace_kinetic}
For each $h:[0,T] \rightarrow \R$ Lipschitz and $f \in L^\infty$ a solution to \eqref{eqn:kinetic}, there are functions $f^+(t,v)$ and $f^-(t,v)$ such that $f^+(t,v) = \chi_{[a^+(t),b^+(t)]}(v)$, $f^-(t,v) = \chi_{[a^-(t),b^-(t)]}(v)$, and
\begin{equation}
\lim_{n \rightarrow \infty} \esssup_{0 < y < \frac{1}{n}} \int_0^T\int_{\R} |f(t,h(t) \pm y, v) - f^\pm(t,v)| \ dvdt = 0.
\end{equation} 
\end{proposition}
We prove first that solutions to \eqref{eqn:kinetic} have traces in a weak topology, in particular, we obtain our trace functions $f^\pm(t,v)$. Next, we employ the blow-up method first introduced for the study of kinetic formulations of conservation laws in \cite{vasseur_kinetic}. That is, we blow-up around a fixed $(t,h(t))$ and use Lemma \ref{lem:averaging1} to conclude strong compactness of the blow-up family. Finally, we identify the strong limit as the weak traces $f^\pm$ and conclude $f^\pm$ are obtained is in the desired topology.

\medskip

\noindent{\underline {\bf {Step 1: Weak traces}}}

\medskip

We begin by showing $f$ has weak traces along a fixed Lipschitz curve $h:[0,T]\rightarrow \R$. Results on weak traces in the regular deformable boundary framework date back to Chen and Frid in \cite{chen_frid}. 
For the sake of completeness, we give here a simple direct proof.
\begin{proposition}\label{prop:trace1}
Say $f \in L^\infty$ is a solution to \eqref{eqn:kinetic}. Then, there exists a full measure set, $\Omega \subset \R$, and functions, $f^+,f^-\in L^\infty([0,T]\times \R)$, with $0\le f^\pm \le 1$, such that for any test function $\Phi\in C^\infty_c([0,T]\times \R)$,
\begin{equation}\label{eqn:lim_desired1}
\lim_{y \rightarrow 0^\pm, y\in \Omega}\int_0^T\int_{\R} f(t, h(t) \pm y,v)\Phi(t,v) \; dvdt = \int_0^T\int_{\R} f^\pm(t,v)\Phi(t,v) \; dvdt
\end{equation} 
\end{proposition}

Let us show the following partial result, which says that the map $y\mapsto (v-\dot{h}(t))f(t,h(t) + y, v)$ has left and right limits, in the sense of distributions. More precisely, we have:
\begin{lemma}\label{lem:trace1}
Suppose $f \in L^\infty$ is a solution to \eqref{eqn:kinetic}. Then, there is a set $\Omega \subset \R$ of full measure and functions $g^\pm \in L^\infty([0,T]\times \R)$ such that for any $\Phi \in L^1([0,T] \times \R)$ with essentially compact support,
\begin{equation}\label{eqn:lim_desired2}
\lim_{y \rightarrow 0^\pm, y\in\Omega} H_\Phi(z) = \lim_{y\rightarrow 0^\pm, y\in \Omega} \int_0^T\int_{\R} [v-\dot{h}(t)]f(t,h(t) + y, v)\Phi(t,v) \;dvdt = \int_0^T\int_{\R} g^\pm(t,v) \Phi(t,v) \;dvdt.
\end{equation}
\end{lemma}

\begin{proof}
We prove \eqref{eqn:lim_desired2} only for right limits. The proof for left limits is identical.
First, fix $\Phi \in C^\infty_c([0,T]\times \R)$ and $\psi\in C^\infty_c(\R)$ with $\|\psi\|_{L^\infty} \le 1$ and take $\Psi(t,x,v) = \Phi(t,v)\psi(x - h(t))$. Since $\Psi$ is compactly supported, Lipschitz in $t$, and smooth in $x$ and $v$, $\Psi$ is an admissible test function for \eqref{eqn:kinetic}. Testing \eqref{eqn:kinetic} with $\Psi$ yields
\begin{equation}
\begin{aligned}
\int_{\R^2\times [0,T]} \partial_t \Phi(t,v)\psi(x - h(t))f(t,x,v) \  &+ \ [v-\dot{h}(t)]\Phi(t,v)\psi^\prime(x-h(t))f(t,x,v) \ ;dvdxdt\\
	& = \int_{\R^2\times [0,T]} \partial_{vv}\Phi(t,v)\psi(x-h(t)) \ d\mu(t,x,v).
\end{aligned}
\end{equation}
Changing variables so that $y = x-h(t)$, we obtain
\begin{equation}
\begin{aligned}
\int_{\R^2\times[0,T]} [v-\dot{h}(t)]\Phi(t,v)\psi^\prime(y)f(t,h(t) + y,v) \;dvdydt = &- \int_{\R^2\times[0,T]} \partial_t \Phi(t,v)\psi(y)f(t,h(t) + y,v) \;dvdydt\\
	&+\int_{\R^2\times[0,T]} \partial_{vv}\Phi(t,v)\psi(x-h(t)) \ d\mu(t,x.v).
\end{aligned}
\end{equation}
Defining $H_{\Phi}(y) \coloneqq \int [v-\dot{h}(t)]f(t,h(t) + y,v)\Phi(t,v) \ dtdv$, since $\|\psi\|_{L^\infty},\|f\|_{L^\infty} \le 1$ and $\mu$ is a finite measure, we obtain
\begin{equation}
\int_{\R} \psi(y) H_{\Phi}(y) \;dy \le C(\mu,\Phi) \|\psi\|_{L^\infty} \le C(\mu,\Phi).
\end{equation}
By a standard characterization of $BV$ functions, $H_{\Phi}(z)$ is of bounded variation and therefore has left and right essential limits for any $y$ and any $\Phi\in C^\infty_c([0,T]\times \R)$. Now, fix $D \subset C^\infty_c([0,T]\times \R)$ a countable dense subset of $L^1([0,T]\times \R)$. Since $D$ is countable, there is a full measure subset of $\Omega$, still denoted $\Omega$, such that
\begin{equation}\label{eqn:trace_convergence1}
\lim_{y \rightarrow 0^+, y \in \Omega} H_{\Phi}(y) \quad\text{exists for any }\Phi\in D.
\end{equation}
Second, we note the functions $\set{(t,v) \mapsto [v-\dot{h}(t)]f(t,h(t) + y, v) \ | \ y\in \R}$ are uniformly bounded in $L^\infty([0,T]\times \R)$. Thus, there is a sequence $y_n \rightarrow 0^+$ with $y_n \in \Omega$ and function $g^+\in L^\infty([0,T]\times \R)$ such that for every $\Phi \in L^1([0,T]\times \R)$,
\begin{equation}\label{eqn:trace_compactness}
\lim_{n\rightarrow \infty} H_{\Phi}(y_n) = \int_0^T\int_{\R} g^+(t,v)\Phi(t,v) \ dvdt, \quad\text{for every }\Phi\in L^1([0,T]\times\R).
\end{equation}
Thus, \eqref{eqn:trace_convergence1} and \eqref{eqn:trace_compactness} together imply
\begin{equation}\label{eqn:trace_convergence2}
\lim_{y \rightarrow 0^+, y\in \Omega} H_{\Phi}(y) = \int_0^T\int_{\R} g^+(t,v)\Phi(t,v) \ dvdt, \quad \text{for every }\Phi \in D.
\end{equation}
Third, we will conclude \eqref{eqn:trace_convergence2} holds for any $\Phi \in L^1([0,T]\times \R)$ with essentially compact support.
Since $f$ is uniformly bounded and $h$ is Lipschitz, we have the uniform (in $y$) bound,
\begin{equation}\label{eqn:Hbound}
|H_\Phi(y)| \le \|f\|_{L^\infty([0,T]\times \R\times \R)}\|v-\dot{h}\|_{L^\infty(supp(\Phi))}\|\Phi\|_{L^1}.
\end{equation}
Therefore, a density argument using \eqref{eqn:trace_convergence2}, $D$ is a dense subset of $L^1([0,T]\times \R)$, and \eqref{eqn:Hbound}, concludes the proof.
\end{proof}

We are now ready to prove Proposition \ref{prop:trace1} by showing $f^\pm(t,v) = \left[v-\dot{h}(t)\right]^{-1}g^\pm(t,v)$:

\begin{proof}[Proof of Proposition \ref{prop:trace1}.]
Again, we prove \eqref{eqn:lim_desired1} only for right limits.
First, we fix $\Omega$ from Lemma \ref{lem:trace1} and note that the functions $\set{y \mapsto f(t,h(t) + y,v) \ | \ y \in \R}$ are uniformly bounded in $L^\infty([0,T]\times \R)$. Thus, it follows that there is a sequence $y_n \rightarrow 0^+$ with $y_n \in \Omega$ and a function $(t,v) \mapsto f^+(t,v)$, for which $f(t,h(t) + y_n,v) \weakstar f^+(t,v)$ in $L^\infty([0,T]\times \R)$. 
Second, for each $\eps > 0$, define the family of domains
\begin{equation}\label{defn:A}
A_{\eps}:= \set{(t,v) \ \biggr| \ 0 \le t \le T, \dot{h}(t) \text{ exists, and } |\dot{h}(t) - v| \le \eps}.
\end{equation}
By Rademacher's theorem and Fubini's theorem, $A_{\eps}$ is measurable with measure $|A_{\eps}| = 2\eps T$. We note that if $\Phi\in L^1([0,T] \times \R)$ has essentially compact support $K$ contained in $[0,T]\times [-L,L]$ with $|K\cap A_{\eps}| = 0$ for some $\eps >0$, applying Lemma \ref{lem:trace1} to $\frac{\Phi(t,v)}{\dot{h}(t) - v}$ and using $f(t,h(t) + y_n,v) \weakstar f^+(t,v)$ in $L^\infty([0,T]\times\R)$,
\begin{equation}
\lim_{y \rightarrow 0^+, y\in\Omega} \int_0^T\int_{\R} f(t,h(t) + y,v) \Phi(t,v) \ dvdt = \int_0^T\int_{\R} f^+(t, v)\Phi(t,v) \ dvdt.
\end{equation}
In other words, for any $L > 0$ and $\eps > 0$, $f(t,h(t) + y,v) \weakstar f^+(t,v)$ in $L^\infty([0,T]\times[-L,L] \setminus A_{\eps})$. 
Third, fix $\Phi \in L^1([0,T]\times \R)$. Then, take a sequence $\eps_n \rightarrow 0^+$ and define $\Phi_n(t,v) = \chi_{[0,T]\times[-n,n]}(1 - \chi_{A_{\eps_n}}(t,v))\Phi(t,v)$. Then,
\begin{equation}
\lim_{y\rightarrow 0^+, y\in \Omega} \int_0^T\int_{\R} \Phi_n(t,v)f(t,h(t)+ y,v)\;dvdt = \int_0^T\int_{\R} \Phi_n(t,v)f^+(t,v)\;dvdt, \quad \text{for each }n\in\mathbb{N}.
\end{equation}
Moreover, since $|A_{\eps}| = 2T\eps$, $\Phi_n \rightarrow \Phi$ in $L^1([0,T]\times \R)$. Thus, for $z_1, z_2 > 0$,
\begin{equation}
\begin{aligned}
&\left|\int_0^T\int_{\R} \Phi(t,v)\left[f(t,h(t)+ z_1,v) - f(t,h(t) + z_2)\right]\ dvdt\right| \\
 &\qquad\qquad\qquad\le2\|\Phi - \Phi_n\|_{L^1} + \left|\int_0^T\int_{\R} \Phi_n(t,v)\left[f(t,h(t)+ z_1,v) - f(t,h(t) + z_2)\right]\;dvdt\right|.
\end{aligned}
\end{equation}
Now, for any $\eta > 0$, pick $n$ sufficiently large so that $\|\Phi - \Phi_n\|_{L^1} < \eta/2$. Then, since for a fixed $n$, the second term is Cauchy, the left hand side is bounded by $\eta$ provided that $z_1,z_2$ are sufficiently small. Combined with $f(t,h(t) + y_n,v) \weakstar f^+(t,v)$ in $L^\infty([0,T]\times \R)$, we obtain
\begin{equation}
\lim_{y\rightarrow 0^+, y\in \Omega} \int_0^T\int_{\R} \Phi(t,v)f(t,h(t)+ y,v)\ dvdt = \int_0^T\int_{\R} \Phi(t,v)f^+(t,v) \ dvdt, \quad \text{for each }\Phi\in L^1([0,T]\times\R)
\end{equation}
and we conclude $f(t,h(t) + y,v) \weakstar f^+(t,v)$ in $L^\infty([0,T]\times \R)$. 
In particular, this means that $f(t,h(t) + y,v)$ obtains its trace function $f^+(t,v)$ in the sense of distributions in $t,v$ and weak star in $L^\infty([0,T]\times\R)$. By lower semicontinuity of norms with respect to weak convergence, $f^+ \in L^\infty$ with $|f^+| \le 1$. Finally, since $\int f^+ \Phi \;dtdv \ge 0$ for any $\Phi(t,v) \ge 0$, $f^+ \ge 0$ almost everywhere.
\end{proof}

\medskip

\noindent{\underline {\bf {Step 2: Strong convergence of blow-ups}}}

\medskip

In this step, we will show that there is a sequence of scales $r_n \rightarrow 0^+$, so that blowing up $f$ around a point $(t,h(t))$ at scales $r_n$, we obtain a strong $L^1_{loc}$ limit, $f^\infty$, simultaneously for almost every point $(t,h(t))$. Moreover, our strong limit $f^\infty$ will solve \eqref{eqn:kinetic} with $\mu = 0$, except possibly along the plane $\{y = 0\}$. The proof consists of using averaging lemmas for the kinetic formulation \eqref{eqn:kinetic} to gain compactness for averages of $f$. Since $f$ is already weakly compact and is almost everywhere the characteristic function of an interval in $v$, this will be sufficient deduce strong compactness for $f$ as well.
\begin{definition}[Blow-up along a curve]
For $(f,\mu)$ a solution pair to \eqref{eqn:kinetic}, $t \mapsto h(t)$ a fixed Lipschitz curve, and $(t,h(t))$ a fixed point on $h$, we define the blow-up $(f_r, \mu_r)$ for $r$ sufficiently small (depending on $t$) by
\begin{equation}\label{defn:rescaling}
\begin{aligned}
f_r(\tau,y,v) &\coloneqq f(t + r\tau,h(t + r\tau) + r y, v)\\
\mu_r(\tau,y,v) &\coloneqq r\mu(t + r\tau,h(t + r\tau) + ry, v).
\end{aligned}
\end{equation}
\end{definition}

Note that the domain of $f_r$ is $[-\frac{t}{r},\frac{T}{r}]\times \R_y \times \R_v$, which converges to $\R_\tau \times \R_y \times \R_v$ as $r \rightarrow 0^+$.
We define also the following sets for $\T > 0$ arbitrary, $D_{\T} = [-\T,\T]\times [-1,1] \times [-L,L]$ and $D_{\T}^+ = D_{\T} \cap \set{y > 0}$ and $D_{\T}^- = D_{\T} \cap \set{y < 0}$, where $L$ is a fixed bound on the support of $f(t,x,\cdot)$.
Definition \ref{defn:rescaling} is meant in the sense of distributions so that, by scaling, $f_r,\mu_r$ almost satisfy \eqref{eqn:kinetic} in the microscopic $(\tau, y, v)$ variables. Finally, fix $J_r$ as the Lipschitz map
$$J_r(\tau,y,v) \coloneqq (t + r\tau, h(t + r\tau) + ry, v),$$
so that $\mu_r(A) = \frac{1}{r} \mu(J_r(A))$ for any $A \subset [-\frac{t}{r},\frac{T}{r}]\times \R_y \times \R_v$.

\begin{proposition}\label{prop:trace2}
For $f$ a solution to \eqref{eqn:kinetic}, there is a sequence $r_n \rightarrow 0$, such that for almost every $t\in [0,T]$, for $f_n = f_{r_n}$, there exist $f_\infty^+$ and $f_\infty^-$, characteristic functions of intervals (in $v$) for almost every $(\tau,y)$, such that for each $\T > 0$,
\begin{equation}
\lim_{n\rightarrow \infty} \int_{D_{\T}^+} |f_\infty^+(\tau,y,v) - f_n(\tau,y,v)| \;dvdyd\tau + \int_{D_{\T}^-} |f_\infty^-(\tau,y,v) - f_n(\tau,y,v)| \;dvdyd\tau = 0.
\end{equation}
Furthermore, the limit functions $f_\infty^\pm$ satisfy the equations
\begin{equation}\label{eqn:limit}
\partial_{\tau} f_\infty^\pm + \left(v - \dot{h}(t)\right)\partial_{y} f_\infty^\pm = 0 \qquad \text{for }(\tau,y,v)\in \bigcup_{\T > 0} D^\pm_{\T}
\end{equation}
\end{proposition}

Let us start with the following lemma concerning the structure of traces for the rescaled solutions:
\begin{lemma}\label{lem:trace2}
For $f$ a solution to \eqref{eqn:kinetic}, the rescaled solutions $f_r$ satisfy
\begin{equation}\label{eqn:rescaled}
\partial_{\tau} f_r + \left(v - \dot{h}(t + r \tau)\right)\partial_yf_r = -\partial_{vv}\mu_r.
\end{equation}
Moreover, $f_r$ has weak traces along the line $y = 0$ given by $f_r^\pm(\tau,v) = f^\pm(t + r\tau, v)$. Furthermore, for any sequence $r_n\rightarrow 0^+$, there is a subsequence still denoted $r_n$ such that for almost every $t\in [0,T]$, the rescaled solutions $f_n = f_{r_n}$ satisfy
\begin{equation}\label{eqn:trace_conv0}
\lim_{n\rightarrow \infty} \|f_n^\pm(\tau,v) - f^\pm(t,v)\|_{L^1_{loc}(\R_\tau \times [-L,L])} = 0.
\end{equation}
\end{lemma}

\begin{proof}
The blow-ups $\{f_r\}_{r > 0}$ satisfy \eqref{eqn:rescaled} because $f$ satisfies \eqref{eqn:kinetic} in the sense of distributions. We prove the convergence \eqref{eqn:trace_conv0} only for the right-hand trace $f^+$. We first note that for $\varphi\in C^\infty_c((-\frac{t}{r},\frac{T}{r})\times (-\infty,\infty))$, Proposition \ref{prop:trace1} combined with the rescaling \eqref{defn:rescaling} guarantees that if $f^+$ denotes the right trace of $f$ along $h(t)$, for a fixed $r > 0$,
\begin{equation}\label{eqn:trace_conv1}
\lim_{y \rightarrow 0^+, y\in r^{-1}\Omega} \int_{\R^2} \left[f_{r}(\tau,y,v) - f^+(t + r\tau,v)\right]\varphi(\tau,v) \;dvd\tau = 0.
\end{equation}
Thus, we conclude $f_r$ has weak traces $f^\pm_r(\tau,v)$ along the line $y = 0$ and $f_r^+(\tau,v) = f^+(t + r\tau,v)$. The convergence in \eqref{eqn:trace_conv0} then follows from the continuity of translations on $L^1$. Namely, for $y$ and $r\in \R$ fixed, let us define
\begin{equation}
F_{r}(t) \coloneqq \int_{-\T}^{\T}\int_{-L}^L f^+(t + r\tau,v) - f^+(t,v) \ d\tau dv.
\end{equation}
Then, $F_{r_n} \rightarrow 0$ in $L^1_t([0,T])$ for each $\T > 0$. By a diagonalization argument, there is a subsequence, still denoted $r_n$, such that $F_{r_n}(t) \rightarrow 0$ for each $\T > 0$ and for Lebesgue almost every $t$, as desired.
\end{proof}

\medskip

Next, we have the following lemma, which shows that we may find a sequence of scales along which the entropy dissipation measure, $\mu$, disappears away from the plane $\{y = 0\}$ as we blow up:
\begin{lemma}\label{lem:trace3}
Suppose $\mu$ is a finite Borel measure. Then, there is a sequence $r_n \rightarrow 0^+$ such that for almost every $t$, $\mu_{r_n}\left(D_{\T}^+ \cup D_{\T}^-\right) \rightarrow 0$ for all $\T > 0$.
\end{lemma}

\begin{proof}
We will prove the statement only for the right domains, i.e. for $D_{\T}^+$. First, using Fubini's theorem in $y$, we compute for a fixed $t,\  r,$ and $\T$ satisfying $T - r\T > t > r\T$,
\begin{equation}
\begin{aligned}
\mu_r\left(D_{\T}^+\right) &= \frac{1}{r}\mu\left(J_{r}\left(D_{\T}^+\right)\right)\\
	&= \frac{1}{r}\mu\left(\set{(t + r\tau, h(t +r\tau) + r y, v) \  \bigg| \ \tau\in [-\T,\T],\ y\in (0,1],\  v\in[-L,L]}\right)\\
	&= \frac{1}{r}\int_{t-r\T}^{t + r\T} \mu_\tau \left(S_\tau \right),
\end{aligned}
\end{equation}
where $\mu_\tau$ denotes $\mu$ restricted to the slice $\{\tau\} \times \R \times [-L,L]$ and, similarly, $S_\tau$ denotes the slice in $\tau$, which is given via
$$S_\tau \coloneqq \{\tau\}\times (h(\tau),h(\tau) + r] \times [-L,L].$$
Thus, integrating in $t$, using Fubini's theorem, and changing variables, yields
\begin{equation}
\begin{aligned}
\int_{r\T}^{T-r\T} \mu_r\left(D_{\T}^+\right) \ dt &= \frac{1}{r}\int_{r\T}^{T-r\T} \int_{t-r \T}^{t+r \T}\mu_{\tau}\left(S_\tau\right) \ d\tau dt\\
	&= \frac{1}{r}\int_{-r\T}^{r\T}\int_{r\T}^{T-r\T} \mu_{t + \tau}\left(S_{t + \tau}\right) \ dtd\tau\\
	&= \frac{1}{r}\int_{-r\T}^{r\T}\int_{r\T+\tau}^{T -r\T+\tau} \mu_{t}\left(S_t\right) \ dtd\tau\\
	&\le \frac{1}{r}\int_{-r\T}^{r\T} \mu\left(\set{(t, h(t) + r y, v) \  \bigg|\ t\in [0,T], \ y\in (0,1],\  v\in[-L,L]}\right) \ d\tau\\
	&= 2\T\mu\left(\set{(t, h(t) + r y, v) \  \bigg|\ t\in [0,T], \ y\in (0,1],\  v\in[-L,L]}\right).
\end{aligned}
\end{equation}
Note that the above integration in $t$ is over $r\T < t < T - r\T$ simply to ensure that rescaling of $\mu$, which we recall depends on $h:[0,T] \rightarrow\R$, is well-defined.
Since the intersection over all $r > 0$ is empty and $\mu$ is a finite measure, the right hand side converges to $0$. Therefore, $\chi_{[r\T,T-r\T]}(t)\mu_r\left(D_{\T}^+\right)\rightarrow 0$ strongly in $L^1([0,T])$ for any $\T > 0$. After extracting subsequences and performing a diagonalization argument in $\T$, there is a sequence $r_n\rightarrow 0^+$ such that for almost every $t \in [0,T]$, $\mu_n\left(D_{\T}^+\right) \rightarrow 0$ for each $\T>0$.
\end{proof}

\medskip

The next lemma is taken directly from \cite{vasseur_kinetic}, but proved in Section \ref{sec:appendix} for the reader's convenience:
\begin{lemma}[\protect{\cite[Lemma 1.1]{vasseur_kinetic}}]\label{lem:compactness}
Suppose that for each $n$, $g_n:\R\times\R \times [-L,L]\rightarrow [0,1]$ satisfies $g_n(\tau,y,\cdot) = \chi_{[a_n(\tau,y),b_n(\tau,y)]}$ for almost every $(\tau,y)$. If we have the convergences
\begin{align}
g_n \weakstar g \qquad &\text{in}\quad L^\infty(\R^3)\\
\int_{-L}^L g_n(\tau,y,v) \;dv \rightarrow \int_{-L}^L g(\tau,y) \;dv \qquad &\text{in}\quad L^1_{loc}(\R^2)\\
\int_{-L}^L vg_n(\tau,y,v) \;dv \rightarrow \int_{-L}^L vg(\tau,y) \;dv \qquad &\text{in}\quad L^1_{loc}(\R^2),
\end{align}
then $g_n \rightarrow g$ in $L^1_{loc}(\R^3)$ and $g(\tau,y,\cdot)$ is almost everywhere the characteristic function of an interval.
\end{lemma}

\medskip

We are now ready to proceed with the proof of Proposition \ref{prop:trace2}:

\begin{proof}[Proof of Proposition \ref{prop:trace2}]
First, we use Lemma \ref{lem:trace3} to pick a sequence $r_n \rightarrow 0^+$ such that for almost every $t$, $$\lim_{n\rightarrow \infty}\mu_n\left(D_{\T}^+\right) = \lim_{n\rightarrow \infty} \mu_n\left(D_{\T}^-\right) = 0, \qquad\text{for each }\T > 0.$$
In particular, we note
$$\sup_n \{\mu_n(D_{\T}^-),\mu_n(D_{\T}^+)\} \lesssim_{\T} 1.$$
Furthermore, defining $a(v) = v - \dot{h}(t)$, by Lemma \ref{lem:trace2}, $f_n$ satisfy
\begin{equation}\label{eqn:kinetic_approx}
\partial_{\tau} f_n + a(v)\partial_y f_n = \left(\dot{h}(t + r_n\tau) -\dot{h}(t)\right)\partial_y f_n - \partial_{vv}\mu_n.
\end{equation}
We wish to apply the compactness result of Lions, Perthame, and Tadmor in Lemma \ref{lem:averaging1} to \eqref{eqn:kinetic_approx}. First, since we have a particularly simple $a$, the non-degeneracy condition \eqref{eqn:nondegenerate} is immediate. Second, we show
\begin{equation}\label{eqn:RHSconvergence}
(1 + \Delta_{\tau,y})^{-1/2}(1 + \partial_{vv})^{-3/2}\left[\left(\dot{h}(t + r_n \tau) -\dot{h}(t)\right)\partial_y f_n - \partial_{vv}\mu_n\right] \rightarrow 0,
\end{equation}
where the convergence is in $L^1_{loc}$ as $n \rightarrow \infty$. By the Lebesgue differentiation theorem, for any $t$ a Lebesgue point of $\dot{h}(t)$,
\begin{equation}
\left\|\dot{h}(t + r_n \tau) -\dot{h}(t)\right\|_{L^2_\tau([-\T,\T])} \rightarrow 0,
\end{equation}
for any $\T > 0$. Therefore, a simple duality argument using $\|f_n\|_{L^\infty} \le 1$ implies
\begin{equation}
\left\|\left(\dot{h}(t + r_n \tau) -\dot{h}(t)\right)\partial_y f_n\right\|_{H^{-1}\left(D_{\T}^\pm\right)} \rightarrow 0,
\end{equation}
for any fixed $t$ a Lebesgue point of $\dot{h}$ and any $\T>0$. Now, since $(1+\partial_{vv})^{-3/2}(1 + \Delta_{\tau,y})^{-1/2}: H^{-1}\left(D_{\T}^{\pm}\right) \rightarrow L^2\left(D_{\T}^\pm\right)$ is a bounded linear operator, we conclude
\begin{equation}
\left\|(1+\partial_{vv})^{-3/2}(1 + \Delta_{\tau,y})^{-1/2}\left(\dot{h}(t + r_n\tau) -\dot{h}(t)\right)\partial_y f_n\right\|_{L^2\left(D_{\T}^\pm\right)} \rightarrow 0.
\end{equation}
On the other hand, from the Morrey's inequality, we obtain the embedding $\mathcal{M}_{loc} \embeds W^{-s,p}_{loc}$ for $s > \frac{3(p-1)}{p}$. Note that for $1 < p < \frac{6}{5}$, we have $\mathcal{M}_{loc} \embeds W^{-\frac{1}{2},p}_{loc}$. Therefore, repeating the above argument yields for each $\T > 0$,
\begin{equation}
\left\|(1 + \Delta_{\tau,y})^{-1/2}(1 + \partial_{vv})^{-3/2}\partial_{vv}\mu_n\right\|_{L^p\left(D_{\T}^\pm\right)} \rightarrow 0.
\end{equation}
and we have established \eqref{eqn:RHSconvergence}.

Thus, using $0 \le f_n \le 1$ we apply Banach-Alaoglu to obtain $f^{\pm}_\infty$ such that $f_n \weakstar f^+_\infty$ in $L^\infty\left(D^+_{\T}\right)$ and $f_n \weakstar f^-_\infty$ in $L^\infty\left(D^-_{\T}\right)$ for each $\T > 0$. 
Now, applying the averaging lemma result, Lemma \ref{lem:averaging1}, we are guaranteed that for any compactly supported $\psi \in L^\infty_v(\R)$, $\int f_n(\cdot,\cdot,v)\psi(v) \ dv \rightarrow \int f^+_\infty(\cdot,\cdot,v)\psi(v)\ dv$ in $L^1(D_{\T}^+)$ and similarly for $f^-_\infty$.
In particular, taking $\psi(v) = \chi_{[-L,L]}(v)$ and $\psi(v) = v\chi_{[-L,L]}(v)$, Lemma \ref{lem:compactness} guarantees for almost every $(\tau,y)$, $f^+_\infty$ and $f^-_\infty$ are characteristic functions of intervals obtained strongly in $L^1\left(D_{\T}^+\right)$ and $L^1\left(D_{\T}^-\right)$, respectively, for each $\T>0$.

Finally, we note that \eqref{eqn:RHSconvergence} guarantees that the limit functions $f^\pm_\infty$ satisfy the desired limiting equations, namely \eqref{eqn:limit}.
\end{proof}

\medskip

\noindent{\underline {\bf {Step 3: Rigidity}}}

\medskip

We now use that along our particular sequence of blow-up scales, the entropy-dissipation measure vanishes away from the plane $\{y = 0\}$ to conclude that our limit, $f_\infty$, is constant on the two strips $D_{\T}^+$ and $D_{\T}^-$. Moreover, we will show that our limit is actually the weak trace and conclude by Proposition \ref{prop:trace2} that the weak traces $f^\pm(t,v)$ are characteristic functions of intervals (in $v$) for almost every $t$.

\begin{proposition}\label{prop:trace3}
Let $f$ be a solution to \eqref{eqn:kinetic}. Then, there is a sequence of scales $r_n \rightarrow 0^+$ such that for almost every $t\in [0,T]$, there are functions $f^\pm_\infty(\tau,y,v) \in L^\infty\left(D_{\T}\right)$ with
$$\lim_{n\rightarrow \infty} \| f_{r_n} - f^\pm_\infty\|_{L^1\left(D_{\T}^\pm\right)} = 0,$$ for each $\T > 0$.
Furthermore, for almost every $t$, and almost every $(\tau,y,v)\in D_{\T}^\pm$, $f^\pm(t, v) = f^\pm_\infty(\tau,y,v)$.
\end{proposition}

\begin{proof}
Pick $r_n \rightarrow 0^+$ given by Proposition \ref{prop:trace2}. Then, up to extracting a subsequence still called $r_n$, both Proposition \ref{prop:trace2} and Lemma \ref{lem:trace2} hold along $r_n$.
For $\varphi\in C^\infty_c([-\T,\T]\times[-L,L])$ and $y > 0$ we define the functional
\begin{equation}
g_\varphi^n(y) = \int_{\R^2} \left[(v - \dot{h}(t + r_n\tau))f_{n}(\tau,y,v) - (v - \dot{h}(t))f^+_\infty(\tau,y,v)\right]\varphi(\tau,v) \;dvd\tau.
\end{equation}
Now, for $\psi\in C^1_c([0,1])$ with $\|\psi\|_{L^\infty} \le 1$, using $f_{n}$ solve \eqref{eqn:rescaled}, we have 
\begin{equation}
\begin{aligned}
\left|\int_{\R} \psi^\prime(y)g_\varphi^n(y) \;dy\right| &\le \int_{\R^3}\left|\psi(y)\partial_{\tau}\varphi(\tau,v)\left[f^+_\infty(\tau,y,v) - f_{n}(\tau,y,v)\right]\right| \ d\tau dvdy\\
 &\quad + \int_{\R^3} |\psi(y)\partial_{vv}\varphi(\tau,v)| \ d\mu_{n}(\tau,y,v).
\end{aligned}
\end{equation}
As $\mu_{r_n}$ is a finite measure for each $n$, $g_\varphi^n(y)$ is of bounded variation on $[0,1]$, with variation bounded via
\begin{equation}
Var(g_\varphi^n) \le \|\partial_{\tau}\varphi\|_{L^\infty}\|f^+_\infty - f_{r_n}\|_{L^1(D_{\T}^+)} + \|\mu_{r_n}\|_{TV}\|\partial_{vv}\varphi\|_{L^\infty}.
\end{equation}
From Proposition \ref{prop:trace2} $f_{r_n} \rightarrow f^+_\infty$ in $L^1(D_{\T}^+)$ as $n\rightarrow \infty$ and $\mu_{r_n} \rightarrow 0$ in total variation by Lemma \ref{lem:trace3} as $n\rightarrow \infty$. It follows that $Var(g_\varphi^n) \rightarrow 0$ and therefore, the boundary value $g_\varphi^n(0^+)$ is bounded via (defined because $g_\varphi^n$ is $BV$),
\begin{equation}
g_\varphi^n(0^+) \le g_\varphi^n(1) + Var(g_\varphi^n) \rightarrow 0. 
\end{equation}
On the other hand, for Lebesgue points of $\dot{h}$, $\dot{h}(t + r_n\tau) \rightarrow \dot{h}(t)$ in $L^1([-\T,\T])$. Note also, Proposition \ref{prop:trace1} and Lemma \ref{lem:trace2} imply $f_\infty^+$ and $f_n$ have a weak traces along the line $y=0$. More specifically, $f_n^+(\tau,v) = f^+(t + r_n\tau, v) \rightarrow f^+(t,v)$ in $L^1_{loc}(\tau,v)$ for almost every $t$, which implies
\begin{equation}
\lim_{n\rightarrow \infty} g_\varphi^n(0^+) = \int_{\R^2} (v - \dot{h}(t))\left[f^+(t,v) - f^+_\infty(\tau,0^+,v)\right]\varphi(\tau,v) \ d\tau dv,
\end{equation}
where $f^+_\infty(\tau,0^+,v)$ denotes the right (weak) trace of $f_\infty^+$ along $y=0$. By the fundamental lemma of the calculus of variations,
\begin{equation}
\left(v - \dot{h}(t)\right)\left[f^+(t,v) - f^+_\infty(\tau,0^+,v)\right] = 0,
\end{equation}
for almost every $(\tau,v)\in \R^2$. Since $(v - \dot{h}(t)) \neq 0$ for almost all $(\tau,v)$, it follows that
\begin{equation}
f^+(t,v) = f^+_\infty(\tau,0^+,v),
\end{equation}
in the same sense. Because the limit equation for $f^+_\infty$, namely \eqref{eqn:limit}, is a transport equation with constant initial data, $f^+_\infty$ is constant. Namely,
\begin{equation}
f^+_\infty\left(\tau,y,v\right) = f^+_\infty\left(\tau + \frac{y}{\left[v - \dot{h}(t)\right]}, 0^+,v\right) = f^+(t,v),
\end{equation}
for each $y > 0$ and almost every $\tau,v \in \R^2$.
\end{proof}

\medskip

\noindent{\underline {\bf {Step 4: Proof of Proposition \ref{prop:trace_kinetic}}}}

\medskip

Fix $t\mapsto h(t)$ a Lipschitz curve and $f$ a solution to \eqref{eqn:kinetic}. Then, by Proposition \ref{prop:trace1}, there exist $f^+$ and $f^-$ weak traces for $f$ along $h(t)$ for which $0 \le f^\pm \le 1$. By Proposition \ref{prop:trace3}, we know that $f^\pm(t,\cdot)$ are realized as the blow-up along a sequence of scales $r_n \rightarrow 0^+$. Thus, by Lemma \ref{lem:compactness}, we conclude $f^\pm(t,\cdot)$ are characteristic functions of intervals (in $v$) for almost every $t$. Now, Proposition \ref{prop:trace_kinetic} follows since weak convergence of characteristic functions to characteristic functions implies strong convergence. More precisely, by Proposition \ref{prop:trace1}, we have for any $\varphi \in L^1([0,T]\times \R \times [-L,L])$, compactly supported,
\begin{equation}\label{eqn:weak_trace}
\lim_{n \rightarrow \infty} \esssup_{0 < y < \frac{1}{n}} \left|\int_0^T\int_{\R} \left[f(t,h(t) + y, v) - f^+(t,v)\right]\varphi(t,v) \;dvdt\right| = 0.
\end{equation}
Finally, recall $f$ and $f^+$ are characteristic functions in $v$, so that taking $\varphi(t,v) = f^+(t,v)$ and $\varphi(t,v) = 1 - f^+(t,v)$ completes the proof.


\section{Partial Regularity}\label{sec:DeGiorgi}

In this section, we prove Proposition \ref{prop:Linfty}, which is a type of partial regularity result for the Riemann invariants of \eqref{eqn:Euler}. The argument is inspired by the first lemma of De Giorgi introduced in \cite{degiorgi} and adapted to the context of scalar conservation laws by Silvestre in \cite[Theorem 6.1]{silvestre}. Note that for \eqref{eqn:Euler} the Riemann invariants appear to decouple and each solve Burger's equation, where the two remain coupled only by the Rankine-Hugoniot jump condition. Thus, if it is known a priori that no shocks are present in a solution, one can apply Silvestre's method directly to recover some form of continuity. However, when it is unknown whether shocks are present, we must necessarily consider both Riemann invariants $\lambda_1$ and $\lambda_2$ simultaneously, which leads to a different usage of the kinetic formulation and different results.

Let us begin with an estimate on the total variation of the entropy dissipation measure $\mu$ on unbounded $v$-intervals using \eqref{eqn:kinetic}:
\begin{lemma}\label{lem:TV_estimate}
Let $f \in L^1(B_2\times \R)$ be a weak solution of \eqref{eqn:kinetic} with $\supp f(t,x) \subset [-L,L]$. Then, for any $0 < r < R$, and any $a \in \R$,
\begin{equation}
\mu\left(B_r \times (-\infty,a]\right) \lesssim_{L} (R - r)^{-1}\int_{B_{R}}\int_{-L}^{a} f \;dvdxdt \qquad \text{and}\qquad \mu\left(B_r \times [a,\infty)\right) \lesssim_{L} (R-r)^{-1}\int_{B_{R}}\int_{a}^L f \;dvdxdt
\end{equation}
\end{lemma}

\begin{remark}
Note that as a consequence of the above lemma, we recover that the support of $\mu$ is contained within $[-L,L]$, which is not immediately apparent from \eqref{eqn:kinetic}. Also, as expected, $\mu$ is not concentrated in velocity.
\end{remark}

\begin{remark}
We are \emph{NOT} able to show the corresponding estimates for $1-f$. More precisely, we \emph{CANNOT} show
\begin{equation}\label{eqn:not_true}
\mu\left(B_r \times [a,b]\right) \lesssim_{L} (R - r)^{-1}\int_{B_{R}}\int_{a}^{b} 1-f \;dvdxdt.
\end{equation}
If we could show a bound akin to \eqref{eqn:not_true}, we would be able to obtain weak-strong stability and uniqueness results for all entropy solutions to \eqref{eqn:Euler}, regardless of size in $L^\infty$ and where the strong solution may contain shocks. In particular, \eqref{eqn:not_true} would allow us to obtain partial regularity results directly for a weak solution $u$ to \eqref{eqn:Euler}, not just on the characteristic fields $\lambda_1(u)$ and $\lambda_2(u)$. However, useful variants of \eqref{eqn:not_true} are likely false, because even if $f \equiv 1$ on some region $A$, we need not have $\mu(A) = 0$. For example, one can construct a shock wave $(u_L,u_R,\sigma_{LR})$, so that the corresponding $f$ is identically $1$ for a portion of the wave front, where $\mu$ does not vanish.
\end{remark}

\begin{proof}[Proof of Lemma \ref{lem:TV_estimate}]
We prove first the bounds for velocity intervals of the form $[a,\infty)$. Fix $b < a$, $\varphi\in C^\infty_c(\R^2)$, and $\zeta_{\eps} \in C^\infty_c(\R)$ such that
$$
\begin{cases} 0 \le \varphi(t,x) \le 1& \\
		\varphi(t,x) = 1 &\text{if }(t,x)\in B_r\\
		\varphi(t,x) = 0 &\text{if }(t,x)\notin B_R\\
		|\nabla \varphi| \le \frac{C}{R-r}&\\
\end{cases} \text{ and }
\begin{cases} 0\le \zeta_{\eps}(v) \le 1,&\\
	\zeta_{\eps}(v) = 0&\text{if }|v-b| > 2\eps^{-1} \text{ or }|v-b| < \eps,\\
	\zeta_{\eps} = 1&\text{if }\eps < |v-b| < \eps^{-1},\\
	|(v-b)\zeta_{\eps}^\prime(v)| \le C,&\\
	|(v-b)^2\zeta_{\eps}^{\prime\prime}(v)| \le C&\\
\end{cases}.$$
Lastly, we define $$\psi(v) = \begin{cases} \frac{(v-b)^2}{2} &\text{if }v > b\\ 0 &\text{otherwise.}\end{cases}$$
We note $\psi \in C^1(\R) \cap C^\infty(\R\setminus\{b\})$ with $\psi^{\prime\prime}(v) = \chi_{(b,\infty)}(v)$. Thus, we test \eqref{eqn:kinetic} with $\varphi(t,x)\psi(v)\zeta_{\eps}(v) \in C^\infty_c(\R^3)$ to obtain
\begin{equation}
\int_{\R^3} \left(\partial_t \varphi + v\partial_x\varphi\right)\psi(v)\zeta_{\eps}(v) f\;dxdtdv = \int_{\R^3} (\psi\zeta_{\eps})^{\prime\prime}\varphi \;d\mu(t,x,v),
\end{equation}
Now, we use 
$$(\psi\zeta_{\eps})^{\prime\prime} = \psi^{\prime\prime}(v)\zeta_{\eps}(v) + 2\psi^\prime(v)\zeta_{\eps}^\prime(v) + \psi(v)\zeta_{\eps}^{\prime\prime}(v)),$$
where the right hand side is bounded pointwise by a constant $C$ independent of $\eps$ and $v$ and the right hand side converges pointwise everywhere to $\chi_{(b,\infty)}(v)$ as $\eps \rightarrow 0+$.  Note, we do not know whether $\mu$ is concentrated in velocity a priori, only that $\|\mu\|_{TV} < \infty$, so we are careful to ensure our bounds hold pointwise everywhere. Thus, by the Lebesgue dominated convergence theorem for the measure $\mu$, we have 
\begin{equation}
\mu(B_1\times [a,\infty)) \le \int_{\R^3} \chi_{(b,\infty)}\varphi \;d\mu(t,x,v) = \int_{\R^3} \left(\partial_t \varphi + v\partial_x\varphi\right)\psi(v) f\;dxdtdv \le \frac{C(L)}{R-r} \int_{B_R}\int_b^L f \;dvdxdt.
\end{equation}
Finally, taking $b\uparrow a$, the monotone convergence theorem yields the desired result for intervals of the form $[a,\infty)$. Redefining $\psi$ so that $$\psi(v) = \begin{cases} \frac{(v-b)^2}{2} &\text{if }v < b\\ 0 &\text{otherwise}\end{cases},$$ the same argument yields the desired bounds on intervals of the form $(-\infty, a]$ as well.
\end{proof}

The following estimate is critical for the proof of Proposition \ref{prop:Linfty}. The estimate encodes the gain of integrability coming from localizing from a larger space-time domain and larger range of velocities to a smaller space-time domain and smaller range of velocities.
\begin{lemma}\label{lem:DeGiorgi_estimate}
Say $f \in L^1(B_2 \times [-L,L])$ is a weak solution of \eqref{eqn:kinetic} with $\supp f(t,x) \subset [-L,L]$. Then, there is a $\theta_0$ (explicitly computable as $\frac{1}{7}$) such that for any $0 < r < R < 2$ and any $-L \le l_2 \le l_1 \le L$, and any $0 < \alpha < \frac{1}{2}$,
\begin{align}
&\|f\|_{L^1(B_r \times [l_1,L])} \lesssim (l_1 - l_2)^{-2\theta_0}(R-r)^{-\theta_0}\|f\|_{L^1(B_R \times [l_2, L])}^{\frac{1 + \theta_0}{2}} \left|\set{ \int_{l_1}^{L}f(t,x,v) > 0\;dv} \cap B_r\right|^{\alpha}\\
&\|f\|_{L^1(B_r \times [-L,l_2])} \lesssim (l_1 - l_2)^{-2\theta_0}(R-r)^{-\theta_0}\|f\|_{L^1(B_R \times [-L, l_1])}^{\frac{1 + \theta_0}{2}} \left|\set{ \int_{-L}^{l_2}f(t,x,v) > 0\;dv} \cap B_r\right|^{\alpha}.
\end{align}
\end{lemma}

\begin{proof}
We will prove only the estimates for sets of the form $B_r \times [l_1,L]$. For this proof, we let $z = (t,x) \in \R^2$ denote space-time. Thus, \eqref{eqn:kinetic} reads as
\begin{equation}\label{eqn:kinetic2}
a(v) \cdot \nabla_z f = -\partial_{vv}\mu,
\end{equation}
where $a(v) = (1,v)$. Now, we fix $f$ a weak solution of \eqref{eqn:kinetic2} with $\supp f(t,x) \subset [-L,L]$ and corresponding entropy dissipation measure $\mu$. 
We also fix $\varphi \in C^\infty_c(\R^2)$ and $\psi\in C^\infty_c(\R)$ satisfying
$$
\begin{cases} 0 \le \varphi(z) \le 1& \\
		\varphi(z) = 1 &\text{if }|z|<r\\
		\varphi(z) = 0 &\text{if }|z| >\widetilde R\\
		|\nabla \varphi| \le \frac{C}{\widetilde R-r}&\\
\end{cases} \text{ and }
\begin{cases} 0\le \psi(v) \le 1,&\\
	\psi(v) = 0&\text{if }v < l_2,\\
	\psi = 1&\text{if } v > l_1,\\
	|\psi^\prime(v)| \le \frac{C}{l_1 - l_2},&\\
	|\psi^{\prime\prime}(v)| \le \frac{C}{(l_1-l_2)^2}&\\
\end{cases},$$
where $\widetilde R = \frac{R + r}{2}$. We define the localized quantities
\begin{align*}
g(z,v) := \psi(v)\varphi(z)f(z,v),\quad \mu_0(z,v) &:= -\left(\varphi(z)\psi^{\prime\prime}(v)\right)\mu(z,v), \quad \mu_1(z,v) := \left(2\varphi(z)\psi^{\prime}(v)\right)\mu(z,v),\\
	&\text{and}\ \mu_2(z,v) := -\left(\varphi(z)\psi(v)\right)\mu(z,v),
\end{align*}
which solve a version of \eqref{eqn:kinetic2} with lower order terms, namely
\begin{equation}
a(v) \cdot \nabla_z g = \partial_{vv}\mu_2 + \partial_v \mu_1 + \mu_0 + \psi f\left[a(v) \cdot\nabla_z \varphi\right].
\end{equation}
Therefore, applying Lemma \ref{lem:averaging2} to $g$, we conclude for any $0 < \theta < \theta_0 = \frac{1}{7}$,
\begin{equation}\label{eqn:estimate1}
\left\|\int g \;dv \right\|_{W^{\theta,\frac{2}{1+\theta_0}}} \lesssim_{\theta,L} \|g\|_{L^2}^{1-\theta_0} \left(\|\mu_2\|_{TV} + \|\mu_1\|_{TV} + \|\mu_0\|_{TV} + \|\psi f(a\cdot \nabla \varphi)\|_{L^1}\right)^{\theta_0}.
\end{equation}
We will upper bound each term on the right hand side individually. First, since $0\le \psi,\varphi,f \le 1$, we have
$$\|g\|_{L^2}^{1-\theta_0} \le \|f\|_{L^1(B_{\widetilde R}\times[l_2,L])}^{\frac{1-\theta_0}{2}}.$$
Second, since $\|a\psi\|_{L^\infty([-L,L])} \le L$, we obtain the bound
\begin{equation}
\|\psi f(a\cdot \nabla \varphi)\|_{L^1} \le \frac{CL}{\widetilde R -r} \|f\|_{L^1(B_{\widetilde R} \times [l_2,L])}.
\end{equation}
Third, for each $i=0,\ 1,\ 2$, we use $\frac{d^i}{dv^i} \psi(v) \le \frac{C}{(l_1 - l_2)^i}$ and Lemma \ref{lem:TV_estimate} to obtain
\begin{equation}
\|\mu_i\|_{TV} \le \frac{C}{(l_1-l_2)^i}\|\mu\|_{TV(B_{\widetilde R} \times [l_2,L])} \le \frac{C}{(l_1-l_2)^i(R - \widetilde R)}\|f\|_{L^1(B_R \times [l_2, L])}.
\end{equation}
Combining our estimates, we conclude
\begin{equation}
\left\|\int \varphi(z)\psi(v)f(z,v) \;dv \right\|_{W^{\theta,\frac{2}{1+\theta_0}}} \lesssim_{L} \frac{1}{(l_2 - l_1)^{2\theta_0}(R-r)^{\theta_0}}\|f\|_{L^1(B_R \times [l_2,L])}^{\frac{1+\theta_0}{2}}.
\end{equation}
Now, by the Sobolev embedding $W^{\theta, \frac{2}{1+\theta_0}}(\R^2) \embeds L^q(\R^2)$ for $q = \frac{2}{1 + (\theta_0 - \theta)}$, we lower bound the left hand side as
\begin{equation}
\left\|\int \varphi(z)\psi(v)f(z,v) \;dv \right\|_{W^{\theta,\frac{2}{1+\theta_0}}} \gtrsim_{q} \left\|\int_{l_1}^L \varphi(z)f(z,v)\;dv \right\|_{L^q}.
\end{equation}
Finally, we use $\varphi\le 1$ and H\"older's inequality, in the form
\begin{equation}
\int_{B_r}\int_{l_1}^L f(z,v) \;dvdz \le \left\|\int_{l_1}^L f(z,v)\;dv \right\|_{L^q}\left|\set{z\in B_r \ \bigg | \ \int_{l_1}^L f(z,v)\;dv > 0}\right|^{\alpha},
\end{equation}
where $\alpha = \frac{q-1}{q} = \frac{1 - (\theta_0 -\theta)}{2}$ which can be made arbitrarily close to $\frac{1}{2}$ by taking $\theta$ arbitrarily close to $\theta_0$, which completes the proof.
\end{proof}

Next, we iterate the gain from Lemma \ref{lem:DeGiorgi_estimate} to prove the following lemma, which is the direct analogue of Proposition \ref{prop:Linfty}, stated for the kinetic equation \eqref{eqn:kinetic}:
\begin{lemma}\label{lem:DeGiorgi_iteration}
Fix $L > 0$ and $-L < a < b< L$. Then for any $f$ solving \eqref{eqn:kinetic} with $supp(f(t,x)) \subset [-L,L]$ with $\essinf_{(t,x)\in B_2}\int f(t,x) \;dv > M$, there is an $\eps_0 = \eps_0(L,M)$, $\widetilde C = \widetilde C(L,M) > 0$, and $\alpha = \frac{1}{21}$ such that, for any $0 < \eps < \eps_0$, setting $\eta = \widetilde{C}\eps^{\alpha}$, if $f$ solves \eqref{eqn:kinetic} with velocities supported in $[-L,L]$, then
\begin{align}\label{eqn:DeGiorgi_L1_control}
\int_{B_2}\int_b^L f(t,x,v) \;dv < \eps &\qquad\text{ implies }\qquad  f(t,x,v) = 0 \quad\text{for}\quad(t,x,v) \in B_1\times[b+\eta, L]\\
\int_{B_2}\int_{-L}^a f(t,x,v) \;dv < \eps &\qquad\text{ implies }\qquad  f(t,x,v) = 0 \quad\text{for}\quad(t,x,v) \in B_1\times[-L,a-\eta].
\end{align}
\end{lemma}

\begin{proof}
Let us define the scales $r_k = 1 + 2^{-k}$, $B_k = B_{r_k}(0)$, and $l_k = \eta(1- 2^{-k})$, for $\eta < \eps_0$. Let us also fix $\theta = \theta_0$ from Lemma \ref{lem:DeGiorgi_estimate} and $\delta = \theta/2$. Then, we define our iteration quantities as
\begin{equation}
U_k = \int_{B_k}\int_{b + l_k}^L f(t,x,v) \;dvdxdt\qquad \text{and} \qquad V_k = \int_{B_k}\int_{-L}^{a-l_k} f(t,x,v) \;dvdxdt.
\end{equation}
Our goal is to show that for an appropriate choices of $\eta$, $\eps_0$, $U_k\rightarrow 0$ and $V_k \rightarrow 0$ provided $U_0,V_0 < \eps < \eps_0$. Now, using Lemma \ref{lem:DeGiorgi_estimate} and $U_k,\ V_k$ are decreasing, we obtain the following inequalities:
\begin{align}\label{eqn:DeGiorgi_ineq_separate}
U_{k+1} &\le \frac{C U_k^{\frac{1+\theta}{2}}}{(l_{k+1} - l_k)^{2\theta}(r_k - r_{k+1})^\theta}\left|\left\{\int_{b+l_{k+1}}^L f(t,x,v) \ dv > 0\right\} \cap B_{k+1}\right|^{\frac{1-\delta}{2}}\\
V_{k+1} &\le \frac{C V_k^{\frac{1+\theta}{2}}}{(l_{k+1} - l_k)^{2\theta}(r_k - r_{k+1})^\theta}\left|\left\{\int^{a-l_{k+1}}_{-L} f(t,x,v) \ dv > 0\right\} \cap B_{k+1}\right|^{\frac{1-\delta}{2}}.
\end{align}
Next, we estimate the measure of the sets appearing on the right hand sides. We note, if $\int_{b+l_{k+1}}^L f(t,x,v)\; dv > 0$, then because $\int_{-L}^L f \;dv > M$ on $B_2$ and $f(t,x, \cdot)$ is the characteristic function of an interval, we must have
$$\int_{b + l_k}^L f(t,x,v)\; dv > \min(M,l_{k+1}-l_k).$$
Thus, by Chebychev's inequality we have
\begin{equation}
\left|\left\{\int_{b+l_{k+1}}^L f(t,x,v) \ dv > 0\right\} \cap B_{k+1}\right| \le \frac{U_k}{\min(M,l_{k+1}-l_k)} = \frac{U_k}{\min(M,\eta 2^{-k-1})}.
\end{equation}
Using a similar argument for $\int_{-L}^{a-l_{k+1}} f(t,x,v)\;dv$, we obtain 
\begin{align}
U_{k+1} &\le \frac{C U_k^{\frac{1+\theta}{2} + \frac{1-\delta}{2}}}{(l_{k+1} - l_k)^{2\theta}(r_k - r_{k+1})^\theta \min(M,l_{k+1}-l_k)^{\frac{1-\delta}{2}}}\\
V_{k+1} &\le \frac{C V_k^{\frac{1+\theta}{2} + \frac{1-\delta}{2}}}{(l_{k+1} - l_k)^{2\theta}(r_k - r_{k+1})^\theta  \min(M,l_{k+1}-l_k)^{\frac{1-\delta}{2}}}.
\end{align}
Now, if $\eta < M$, we have $\min(M, l_{k+1}-l_k) = l_{k+1}-l_k$ for each $k$, and
\begin{align}
U_{k+1} &\le \frac{C U_k^{1 + \frac{\theta}{4}}}{(l_{k+1} - l_k)^{\frac{1}{2} + \frac{7\theta}{4}}(r_k - r_{k+1})^\theta} = \frac{C2^{(k+1)\left[\frac{11\theta}{4} + \frac{1}{2}\right]}U_k^{1 + \frac{\theta}{4}}}{\eta^{\frac{7\theta}{4} + \frac{1}{2}}}\\
V_{k+1} &\le \frac{C V_k^{1 + \frac{\theta}{4}}}{(l_{k+1} - l_k)^{\frac{1}{2} + \frac{7\theta}{4}}(r_k - r_{k+1})^\theta} = \frac{C2^{(k+1)\left[\frac{11\theta}{4} + \frac{1}{2}\right]}V_k^{1 + \frac{\theta}{4}}}{\eta^{\frac{7\theta}{4} + \frac{1}{2}}}.
\end{align}
Finally, setting $\alpha^{-1} = \frac{7\theta+2}{\theta}$, we obtain
\begin{equation}\label{eqn:DeGiorgi_iteration}
\frac{U_{k+1}}{\eta^{\frac{1}{\alpha}}} \le C2^{(k+1)\left[\frac{11\theta}{4} + \frac{1}{2}\right]} \left(\frac{U_k}{\eta^\frac{1}{\alpha}}\right)^{1+\frac{\theta}{4}} \quad \text{and} \quad \frac{V_{k+1}}{\eta^{\frac{1}{\alpha}}} \le C2^{(k+1)\left[\frac{11\theta}{4} + \frac{1}{2}\right]} \left(\frac{V_k}{\eta^\frac{1}{\alpha}}\right)^{1+\frac{\theta}{4}}.
\end{equation}
Therefore, \eqref{eqn:DeGiorgi_iteration} implies that there is a sufficiently large constant $\widetilde C$ such that if 
$$\frac{U_0}{\eta^{\frac{1}{\alpha}}} = \frac{\|f\|_{L^1(B(2)\times[b,L])}}{\eta^{\frac{1}{\alpha}}} < \frac{\eps}{\eta^{\frac{1}{\alpha}}} \le \frac{1}{\widetilde C^{\frac{1}{\alpha}}},$$
and similarly for $V_0$, then
$$\lim_{k\rightarrow \infty} U_k = \lim_{k\rightarrow \infty} V_k = 0.$$
Choosing $\eta = \widetilde C\eps^\alpha$ and $\eps_0 = (M/\widetilde{C})^{1/\alpha}$ and noting $\alpha = \frac{1}{21}$ completes the proof.
\end{proof}

We are now ready to prove Proposition \ref{prop:Linfty}:
\begin{proof}[Proof of Proposition \ref{prop:Linfty}]
We first note that $u = (\rho,m)$ satisfies $\|\rho\|_{L^\infty(B_2)} + \|\rho/m\|_{L^\infty(B_2)} \le \Gamma$ so the corresponding $f$ solving \eqref{eqn:kinetic}, has a uniform velocity bound $L = 2\Gamma$ on $B_2$. Now, we wish to apply Lemma \ref{lem:DeGiorgi_iteration} to this $f$. Thus, we note $\rho(t,x) \ge M$ for almost every $(t,x) \in B_2$ exactly implies $\int f\; dv \ge M$. Also, since $\overline{u} = (\overline{\rho},\overline{m})$, is a constant solution to \eqref{eqn:Euler}, the corresponding (constant) kinetic function is $\chi_{[a,b]}$ for $a = \frac{\overline{m}}{\overline{\rho}} - \frac{\overline{\rho}}{2} = \lambda_1(\overline{u})$ and $b = \frac{\overline{m}}{\overline{\rho}} + \frac{\overline{\rho}}{2} = \lambda_2(\overline{u})$.
Since $f$ is given via the formula,
$$f(t,x,v) = \chi_{\left[\frac{m}{\rho}-\frac{\rho}{2}, \frac{m}{\rho}+\frac{\rho}{2}\right]}(v) = \chi_{[\lambda_1(u),\lambda_2(u)]}(v),$$
we have
\begin{equation}
\int_{B_2} (\lambda_1(\overline{u}) - \lambda_1(u))_+ = \int_{B_2}\int_{-L}^a f(t,x,v) \;dvdxdt \qquad \text{and} \qquad \int_{B_2} (\lambda_2(u) - \lambda_2(\overline{u}))_+ = \int_{B_2}\int_b^L f(t,x,v) \;dvdxdt.
\end{equation}
Therefore, for $0 < \eps < \eps_0$, where $\eps_0 = \eps_0(L,M)$ is from Lemma \ref{lem:DeGiorgi_iteration}, we have 
\begin{align}
&\text{if }\quad\|\left(\lambda_1(\overline{u}) - \lambda_1(u)\right)_+\|_{L^1(B_2)} \le \eps \qquad \text{ then }\quad f(t,x,v) = 0, \ (t,x) \in B_1, \ v < a - \widetilde{C}\eps^\alpha,\\
&\text{if }\quad\|\left(\lambda_2(u) - \lambda_2(\overline{u})\right)_+\|_{L^1(B_2)} \le \eps \qquad \text{ then }\quad f(t,x,v) = 0, \ (t,x) \in B_1, \ v > b + \widetilde{C}\eps^\alpha.
\end{align}
Thus, we compute pointwise for $(t,x) \in B_1$,
\begin{align}
&\left(\lambda_1(\overline{u}) - \lambda_1(u(t,x))\right)_+ = \int_{-L}^a f(t,x,v)\; dv = \int_{a-\widetilde{C}\eps^{\alpha}}^a f(t,x,v)\; dv \le \widetilde{C}\eps^{\alpha}\\
&\left(\lambda_2(u(t,x)) - \lambda_2(\overline{u})\right)_+ = \int_{b}^L f(t,x,v)\; dv = \int_b^{b+\widetilde{C}\eps^{\alpha}} f(t,x,v)\; dv \le \widetilde{C}\eps^{\alpha},
\end{align}
which completes the proof.
\end{proof}


\section{Semicontinuity}\label{sec:semicontinuity}

In this section, we present our first application of our partial regularity result and prove Theorem \ref{thm:cont}. We begin by recalling several notations and definitions for versions of $L^\infty$ functions. First, for $g\in L^1_{loc}(\R^n)$, we say that $x$ is a point of $VMO$ (or Vanishing Mean Oscillation) for a function $g$, if
\begin{equation}
\lim_{r\rightarrow 0^+}\frac{1}{|B_r(x)|}\int_{B_r(x)} \left|g(y) - \fint_{B_r(x)} g(z)\;dz\right| \;dy = 0.
\end{equation}
Note, that this is a slight relaxation of the notion of a Lebesgue point, which further requires 
$$\lim_{r\rightarrow 0^+}\fint_{B_r(x)} g(z)\;dz = g(x).$$
Second, we recall the upper and lower semicontinuous envelopes, which we denote by $\overline{g}$ and $\underline{g}$ are defined via
\begin{equation}
\overline{g}(x) = \lim_{r\rightarrow\infty}\esssup_{y\in B_r} g(y) \qquad \text{and} \qquad \underline{g}(x) = \lim_{r\rightarrow\infty}\essinf_{y\in B_r} g(y).
\end{equation}
One can check that for $g\in L^\infty$, $\overline{g}$ is upper semicontinuous and $g\le \overline{g}$ almost everywhere and, similarly, $\underline{g}$ is lower semicontinuous and $\underline{g} \le g$ almost everywhere. Finally, we pick the following version for the scalar functions we study:
\begin{equation}
\hat{g}(x) = \liminf_{r\rightarrow \infty} \fint_{B_r(x)} g(y)\;dy.
\end{equation}
With this in mind, we will deduce Theorem \ref{thm:cont} immediately from the following slightly more precise claim:
\begin{proposition}\label{prop:cont}
Suppose $u:\R^+ \times \R \rightarrow \R^2$ is a bounded entropy solution to \eqref{eqn:Euler}. Then, for any $(t,x)$ a point of $VMO$ for $\lambda_1(u)$ and $\lambda_2(u)$ with $\underline{\rho}(t,x) \neq 0$, $(t,x)$ is a Lebesgue point of $\rho$, $m$, $\lambda_1(u)$, and $\lambda_2(u)$. Moreover, 
\begin{equation}
\hat{\rho}(t,x) = \overline{\rho}(t,x), \quad \hat{\lambda}_1(t,x) = \underline{\lambda}_1(t,x), \quad \hat{\lambda}_2(t,x) = \overline{\lambda}_2(t,x),
\end{equation}
and
\begin{equation}
\hat{m}(t,x) = \begin{cases}
	&\overline{m}(t,x) \text{ if } (t,x)\in\{\underline{\lambda}_1 \ge 0 \}\\
	&\underline{m}(t,x) \text{ if }(t,x)\in\{ \overline{\lambda}_2 \le 0 \}
	\end{cases}.
\end{equation}
\end{proposition}

\medskip

\noindent \underline {\bf {The Riemann Invariants}}

\medskip

Now, we fix $u = (\rho,m)$ a bounded entropy solution and $(t,x)$ a point of $VMO$ of $\lambda_1(t,x) \coloneqq \lambda_1(u(t,x))$ and $\lambda_2(u)\coloneqq \lambda_2(u(t,x))$ such that $\underline{\rho}(t,x) \neq 0$.
We will prove $(t,x)$ is a Lebesgue point for $\lambda_1$ and $\lambda_2$ with value $\underline{\lambda}_1$ or $\overline{\lambda}_2$, respectively. To this end, we pick any sequence of scales $r_n \rightarrow 0^+$. Then, using $u$ is bounded, we conclude there is a subsequence $r_{n_k}$ along which $\fint_{B_{r_{n_k}}(t,x)} \lambda_1(\tau,y)\;d\tau dy \rightarrow \Lambda$. Note, since $(t,x)$ is $VMO$ for $\lambda_1$, along $r_{n_k}$ we have
\begin{equation}
\lim_{k\rightarrow \infty} \frac{1}{|B_{r_{n_k}}(t,x)|} \int_{B_{r_{n_k}}(t,x)} |\lambda_1(u(\tau,y)) - \Lambda| \;d\tau dy = 0.  
\end{equation}
Consider the blow-up sequence $u_k(\tau,y) \coloneqq u((t,x) + r_{n_k}(\tau,y))$. Note that by translation invariance and scaling, each blow-up $u_k$ is a solution to \eqref{eqn:Euler} satisfying $\|\rho_k\|_{L^\infty} + \|\frac{m_k}{\rho_k}\|_{L^\infty} \le \Gamma$, where $\Gamma = \|\rho\|_{L^\infty} + \|\frac{m}{\rho}\|_{L^\infty}$. Moreover, $\rho_k \ge \frac{\underline{\rho}(t,x)}{2}$ almost everywhere on $B_2$ for each $k$ sufficiently large.  Thus, for $k$ sufficiently large so that $\|\lambda_1(u_k) - \Lambda\|_{L^1(B_2)} < \eps_0$, where $\eps_0$ is as in Proposition \ref{prop:Linfty}, we conclude
\begin{equation}
\esssup_{B_1} \left[\Lambda - \lambda_1(u_k)\right]_+ \lesssim \|\Lambda - \lambda_1(u_k)\|_{L^1(B_{2})}^\alpha \rightarrow 0.
\end{equation}
Therefore, on one hand, we conclude
\begin{equation}
\Lambda \le \underline{\lambda}_1(t,x).
\end{equation}
On the other hand, from the definition of $\Lambda$, we have
\begin{equation}
\Lambda = \lim_{k\rightarrow \infty} \fint_{B_{r_{n_k}}(t,x)}\lambda_1(\tau,y)\;d\tau dy \ge \lim_{k\rightarrow \infty} \essinf_{(\tau,y)\in B_{r_{n_k}}(t,x)} \lambda_1(\tau,y) = \underline{\lambda}_1(t,x).
\end{equation}
From the uniqueness of the limit, we conclude $(t,x)$ is a Lebesgue point of $\lambda_1$ with value $\underline{\lambda}_1(t,x)$. The same argument for $\lambda_2$ yields $(t,x)$ is a Lebesgue point of $\lambda_2$ with value $\overline{\lambda}_2(t,x)$.
The definitions of $\hat{\lambda}_1$ and $\hat{\lambda}_2$ then guarantee $\hat{\lambda}_1(t,x) = \underline{\lambda}_1(t,x)$ and $\hat{\lambda}_2(t,x) = \overline{\lambda}_2(t,x)$.

\medskip

\noindent \underline {\bf {The Density}}

\medskip

The main tool for analyzing $\rho$ is the algebraic identity, $\rho = \lambda_2(u) - \lambda_1(u)$. Thus, we immediately see that $(t,x)$ a point of $VMO$ of $\lambda_1(u)$ and $\lambda_2(u)$ implies $(t,x)$ is a point of $VMO$ for $\rho$. Also, since $(t,x)$ is a Lebesgue point for $\lambda_1$ and $\lambda_2$, we find
\begin{equation}
\hat{\rho} = \lim_{r\rightarrow 0^+} \fint_{B_r(t,x)} \rho =  \lim_{r\rightarrow 0^+} \fint_{B_r(t,x)} \left(\lambda_2 - \lambda_1\right) = \overline{\lambda}_2(t,x) - \underline{\lambda}_1(t,x),
\end{equation}
and $(t,x)$ is a Lebesgue point for $\rho$. Next, a few basic inequalities imply
\begin{equation}
\lim_{r\rightarrow 0^+}\fint_{B_r(t,x)}\rho \le \lim_{r\rightarrow 0^+}\esssup_{y\in B_r(t,x)} \rho(y) \le \lim_{r\rightarrow 0^+}\left[\esssup_{y \in B_r(t,x)} \lambda_2(y) - \essinf_{y\in B_r(t,x)} \lambda_1(y)\right] = \overline{\lambda}_2(t,x) - \underline{\lambda}_1(t,x).
\end{equation}
Thus, we conclude
\begin{equation}
\hat{\rho}(t,x) = \overline{\rho}(t,x) = \overline{\lambda}_2(t,x) - \underline{\lambda}_1(t,x).
\end{equation}

\medskip

\noindent \underline {\bf {The Momentum}}

\medskip

The main tool for analyzing $m$ is the algebraic identity, $m = \frac{\lambda_2^2 - \lambda_1^2}{2}$. Just as for the density, we see $(t,x)$ a point of $VMO$ of $\lambda_1$ and $\lambda_2$ implies $(t,x)$ is a Lebesgue point for $m$ with
\begin{equation}
\hat{m} = \lim_{r\rightarrow 0^+} \fint_{B_r(t,x)} m =  \lim_{r\rightarrow 0^+} \frac{1}{2|B_r(t,x)|}\int_{B_r(t,x)} \lambda_2^2(y) - \lambda_1^2(y) \;dy = \frac{\overline{\lambda}_2^2(t,x) - \underline{\lambda}_1^2(t,x)}{2}.
\end{equation}
Now, we have two cases. On one hand, if $\underline{\lambda}_2(t,x)\ge \underline{\lambda}_1(t,x) \ge 0$, then
\begin{equation}
\lim_{r\rightarrow 0^+}\fint_{B_r(t,x)} m \le \lim_{r\rightarrow 0^+}\esssup_{y\in B_r(t,x)} m(y) \le \frac{1}{2} \lim_{r\rightarrow 0^+}\left[\esssup_{y \in B_r(t,x)} \lambda_2^2(y) - \essinf_{y\in B_r(t,x)} \lambda_1^2(y)\right] = \frac{\overline{\lambda}_2^2(t,x) - \underline{\lambda}_1^2(t,x)}{2}.
\end{equation}
In this case, $\hat{m}(t,x) = \overline{m}(t,x)$. On the other hand, if $\overline{\lambda}_1(t,x) \le \overline{\lambda}_2(t,x) \le 0$, then
\begin{equation}
\lim_{r\rightarrow 0^+}\int_{B_r(t,x)}m \ge \lim_{r\rightarrow 0^+}\essinf_{y\in B_r(t,x)} m(y) \ge \frac{1}{2} \lim_{r\rightarrow 0^+}\left[\essinf_{y \in B_r(t,x)} \lambda_2^2(y) - \esssup_{y\in B_r(t,x)} \lambda_1^2(y)\right] = \frac{\overline{\lambda}_2^2(t,x) - \underline{\lambda}_1^2(t,x)}{2}.
\end{equation}
In this case, $\hat{m}(t,x) = \underline{m}(t,x)$.

\section{Improved Traces}\label{sec:ODE}

In this section, we present our second application of our partial regularity result and prove Theorem \ref{thm:improved_traces}. To use Proposition \ref{prop:Linfty}, we need some decay of the $L^1$ oscillations as we blow up around a fixed point $(t,h(t))$. In the proof of Theorem \ref{thm:cont}, this is provided for almost every $(t,x) \in \R^+ \times \R$ by Lebesgue's Differentiation Theorem. However, $\set{(t,h(t)) \ | \ t\in\R^+}$ is a Lebesgue null set of $\R^+\times \R$. Therefore, we first use the traces of Theorem \ref{thm:trace} to obtain $L^1$-convergence of the blow-ups.

\medskip

\noindent \underline {\bf {Step 1: $L^1$-Convergence of Blow-ups}}

\medskip

We use the strong trace property provided by Theorem \ref{thm:trace}, which was established in Section \ref{sec:trace} to ensure blow-ups around points $(t,h(t))$ converge in $L^1$. Note that instead of blowing up around points $(t,h(t))$ along the curve $h(t)$ as was done in Section \ref{sec:trace}, we will take a blow-up that preserves equation \eqref{eqn:Euler}.
\begin{lemma}\label{lem:DeGiorgi_L1}
Suppose $u:\R^+ \times \R \rightarrow \R^2$ with $u\in L^\infty$ satisfies the strong trace property in the sense of Definition \ref{defn:strong_traces}. Then, for any $h:[0,T]\rightarrow \R$ Lipschitz and for any sequence $r_k \rightarrow 0^+$, there is a subsequence, still denoted $r_k$, such that:\newline\newline for almost every $t$, the half-spaces $H_t^+ \coloneqq \{(\tau,y) \ | \ y \ge \dot{h}(t)\tau \}$ and $H_t^- \coloneqq \{(\tau,y) \ | \ y \le \dot{h}(t)\tau \}$ are well-defined and for each compact $K^\pm \subset H_t^\pm$,
\begin{equation}\label{eqn:scales}
\lim_{k\rightarrow\infty} \int_{K^\pm} \left|u(t + r_k\tau,h(t) + r_k y) - u^\pm(t)\right| \;d\tau dy = 0.
\end{equation}
\end{lemma}

\begin{proof}
We will prove the statements only for the right trace. First, we note that the strong trace property of $u$ and continuity of translations on $L^1$ together imply for each $\T, Y > 0$,
\begin{equation}\label{eqn:DeGiorgi_trace_scale1}
\lim_{r \rightarrow 0^+} \int_0^T\int_{-\T}^{\T}\int_0^Y \left|u(t + r\tau, h(t + r \tau) + ry) - u^+(t + r\tau)\right| \;dyd\tau dt = 0.
\end{equation}
Next, we note
\begin{equation}\label{eqn:DeGiorgi_trace_scale2}
\lim_{r\rightarrow 0^+} \int_0^T\int_{-\T}^{\T}|u^+(t + r\tau) - u^+(t)| \;d\tau dt = 0.
\end{equation}
Also, taking $\gamma_r = \frac{h(t + r\tau) - h(t)}{r}$, $U_1^r(t,\tau,y) = \left|u(t+r\tau, h(t +r\tau) + r y) - u^+(t)\right|$, and\newline $U_2^r(t,\tau,y) = \left|u(t+r\tau, h(t) + ry) - u^+(t)\right|$, by a change of variables and the fundamental theorem of calculus, we obtain the bound
\begin{equation}\label{eqn:DeGiorgi_trace_scale3}
\begin{aligned}
&\left|\int_0^T\int_{-\T}^{\T}\left(\int_0^Y U_1^r(t,\tau,y) \;dy - \int_{-\dot{h}(t)\tau}^{Y-\dot{h}(t)\tau} U_2^r(t,\tau,y)  \;dy\right)\;d\tau dt\right|\\
	&\qquad\qquad= \left|\int_0^T\int_{-\T}^{\T}\left(\int_{\gamma_r}^{Y + \gamma_r} U_2^r(t,\tau,y) -\int_{\dot{h}(t)\tau}^{Y + \dot{h}(t)\tau} U_2^r(t,\tau,y)\;dy\right)\;d\tau dt\right|\\
	&\qquad\qquad\le 4\|u\|_{L^\infty}\int_0^T\int_{-\T}^{\T} |\gamma_r - \dot{h}(t)\tau| \;d\tau dt\\
	&\qquad\qquad\le 4\|u\|_{L^\infty}\int_0^T\int_{-\T}^{\T} \frac{1}{r}\int^{t+r\tau}_t \left|\dot{h}(s) - \dot{h}(t)\right| \;dsd\tau dt,
\end{aligned}
\end{equation}
where the right hand side converges to $0$ as $r \rightarrow 0^+$ by the Lebesgue Differentiation Theorem. Combining \eqref{eqn:DeGiorgi_trace_scale1}, \eqref{eqn:DeGiorgi_trace_scale2}, and \eqref{eqn:DeGiorgi_trace_scale3}, we obtain for any $\T > 0$, $Y > 0$,
\begin{equation}
\lim_{r \rightarrow 0^+} \int_0^T\int_{-\T}^{\T}\int_{\dot{h}(t)\tau}^{Y+\dot{h}(t)\tau} \left|u(t + r\tau,h(t) + r y) - u^+(t)\right| \;dyd\tau dt = 0.
\end{equation}
Defining $H_t^+ \coloneqq \set{(\tau, y) \ | \ y \ge \dot{h}(t)\tau}$, for any sequence $r \rightarrow 0^+$, by a diagonalization argument, we can find a subsequence along which \eqref{eqn:scales} holds.
\end{proof}

\medskip

\noindent \underline {\bf {Step 2: Uniform Convergence of Blow-ups}}

\medskip

In this step, we combine the $L^1$-convergence result of the previous step with the partial regularity result of Proposition \ref{prop:Linfty} to obtain for almost every $t$ the blow-ups converge away from a shock. In local $(\tau, y)$ coordinates, a shock can only form at the boundary of $H^+_t$, i.e. along the line $y = \dot{h}(t)\tau$.
\begin{lemma}\label{lem:DeGiorgi}
Suppose $u = (\rho, m)$ is a bounded entropy solution to \eqref{eqn:Euler} and $\essinf \rho > 0$. Then, for any $h:[0,T]\rightarrow \R$ a Lipschitz curve, for almost every $t$,
\begin{equation}\label{eqn:convergence_trace}
\lim_{r\rightarrow 0^+} \esssup_{(\tau,y)\in B_{r}(0)} \left( \left[\min(\lambda_1^+(t),\lambda_1^-(t)) - \lambda_1(t + \tau, h(t) + y)\right]_+ + \left[\lambda_2(t + \tau, h(t) + y) - \max(\lambda_2^-(t),\lambda_2^+(t))\right]_+ \right) = 0.
\end{equation} 
Moreover, for any fixed subset $K^\pm$ of $\R^2$ compactly contained in $H^\pm_t$ for almost every $t$, for $g = - \lambda_1, \ \lambda_2, \text{ or } \rho$,
\begin{align}\label{eqn:convergence}
\lim_{r\rightarrow 0^+}\int_0^T\esssup_{(\tau,y) \in K^\pm} \left[g(t + r\tau, h(t) + ry) - g^\pm(t)\right]_+ \;dt= 0.
\end{align}
\end{lemma}

\begin{proof}
Let us fix $u$ an entropy solution to \eqref{eqn:Euler} satisfying $\|\rho\|_{L^\infty} + \|m/\rho\|_{L^\infty} \le \Gamma$ and $0 < M \le \rho$. Now since $M, \Gamma$ are fixed, let us recall $\eps_0 = \eps_0(M) > 0$, $\tilde C(\Gamma, M)$, and $0 < \alpha < 1$ as in the statement of Proposition \ref{prop:Linfty} are fixed.

First, let us note that for any $r > 0$, the rescaling $\tilde u(\tau, y) \coloneqq u(t + r\tau, x + ry)$ is a solution to \eqref{eqn:Euler} satisfying the same bounds. Moreover, for any fixed threshold $\Lambda > 0$, if
$$\frac{1}{r^2}\|[\lambda_1(u) - \Lambda\|_{L^1(B_{2r}(0))} = \|[\lambda_1(\tilde u) - \Lambda]_+\|_{L^1(B_2(0))} < \eps_0,$$ we can apply Proposition \ref{prop:Linfty} to $\tilde u$ and obtain
\begin{equation}\label{eqn:DeGiorgi_rescaled}
\begin{aligned}
\sup_{(\tau, y) \in B_r(t,x)} [\lambda_1(\tau,y) - \Lambda]_+ &= \sup_{(\tau, y) \in B_1(0)} [\lambda_1(t + r\tau , x + ry) - \Lambda]_+\\
	&\le \tilde C(\Gamma, M) \|\lambda_1(t + r\tau, x + ry) - \Lambda\|_{L^1_{\tau,y}(B_2(0))}^\alpha = \frac{\tilde C(\Gamma,M)}{r^{2\alpha}} \|\lambda_1 - \Lambda\|_{L^1(B_{2r}(t,x))}^\alpha.
\end{aligned}
\end{equation}

Second, let us show \eqref{eqn:convergence_trace} for $\lambda_1$. By Lemma \ref{lem:DeGiorgi_L1}, we have a sequence of scales $r_k \rightarrow 0^+$ such that for almost every $t$, the half-spaces $H_t^+ = \set{(\tau,y) \ \big | \ y \ge \tau\dot{h}(t)}$ are defined, $\set{y > 0, \tau = 0} \subset H_t^+$, and for which the blow-ups $u_k:= u(t + r_k\tau, h(t)+r_k y)$ satisfy $u_k \rightarrow u^\pm(t)$ in $L^1_{loc}(H^\pm_t)$. Using $\lambda_1$ is continuous on the state space, by a simple compactness argument implies $\lambda_1(u_k) \rightarrow \lambda_1(u^\pm(t))$ in $L^1_{loc}(H_t^\pm)$. In particular, for almost every $t \in [0,T]$,
\begin{equation}
\lim_{k\rightarrow \infty} \int_{B_2} \left[\min(\lambda_1^+(t),\lambda_1^-(t)) - \lambda_1(u_k)\right]_+ \;dyd\tau = 0.
\end{equation}
Since each $u_k$ is a rescaling of $u$, we use the above argument with $k$ sufficiently large to obtain
\begin{equation}
\begin{aligned}
\lim_{k\rightarrow \infty} \esssup_{(\tau,y)\in B_{r_k}(0)} [\min(\lambda_1^+(t),\lambda_1^-(t)) - &\lambda_1(u(t + \tau, h(t) + y))]_+\\
	&=  \lim_{k\rightarrow \infty} \esssup_{(\tau,y) \in B_1} [\min(\lambda_1^+(t),\lambda_1^-(t)) - \lambda_1(u_k(\tau,y))]_+\\
	&\le \lim_{k\rightarrow \infty} \tilde C(\Gamma,M)\|[\min(\lambda_1^+(t),\lambda_1^-(t)) - \lambda_1(u_k)]_+\|_{L^1(B_2)}^\alpha = 0.
\end{aligned}
\end{equation}
Finally, since the supremum over a descending sequence of sets is monotone, we obtain \eqref{eqn:convergence_trace}. Now, the proof is identical for $\lambda_2$.

Third, let us show \eqref{eqn:convergence} for $\lambda_1$ and some fixed $K^+$. Indeed, fix $r_n \rightarrow 0+$. Then, by Lemma \ref{lem:DeGiorgi_L1}, for any subsequence, there is a further subsequence denoted $r_{n_k}$ along which $u_k :=(u(t + r_{n_k}\tau, h(t) + r_{n_k}y))$ converges to $u^+(t)$ in $L^1_{loc}(H^+_t)$ for almost every $t$. In particular, we see $\|\lambda_1(u_k) - \lambda_1(u^+(t))\|_{L^1(K^+)} \rightarrow 0$ for almost every $t$.

Suppose $t_0$ is such a time and $K^+ = B_R(\tau_0,y_0) \subset B_{2R}(\tau_0, y_0)$ is compactly contained in $H^+_t$. Then, \eqref{eqn:DeGiorgi_rescaled} implies for $k$ sufficiently large
\begin{equation}\label{eqn:convergence2}
\sup_{(\tau, y) \in B_R(\tau_0,y_0)} [\lambda_1(u_k(\tau,y)) - \lambda_1(u^+(t_0))]_+ \le \frac{\tilde C(\Gamma,M)}{R^{2\alpha}} \|\lambda_1(u_k) - \lambda_1(u^+(t_0))\|_{L^1(B_{2R}(\tau_0,y_0))}^\alpha \rightarrow 0.
\end{equation}
On the other hand, if $K^+$ is a general compact subset, we can cover $K^+$ with balls $B_R(\tau_0,y_0)$ such that $B_{2R}(\tau_0,y_0) \subset H^+_{t_0}$, extract a finite subcover, and apply the above result for balls. Therefore, noting $\lambda_1$ is bounded and applying the Lebesgue dominated convergence theorem, we conclude that
$$\lim_{k\rightarrow 0^+}\int_0^T \esssup_{(\tau,y) \in K^+}\left[\lambda_1^+(t) - \lambda_1(t + r_{n_k}\tau, h(t) + r_{n_k}y)\right]_+ \;dt = 0.$$
Since the subsequence was arbitrary, we conclude that the limit holds along the entire sequence and hence
$$\lim_{r\rightarrow 0^+}\int_0^T \esssup_{(\tau,y) \in K^+}\left[\lambda_1^+(t) - \lambda_1(t + r\tau, h(t) + ry)\right]_+ \;dt = 0,$$
as desired. The proofs for left traces and $\lambda_2$ are identical. Finally, the proofs for $\rho$ follow by noting $\rho = \lambda_2 - \lambda_1$ almost everywhere on $\R^+\times \R$.
\end{proof}

\medskip

\noindent \underline {\bf {Step 3: Proof of Theorem \ref{thm:improved_traces}}}

\medskip

We are now ready to deduce Theorem \ref{thm:improved_traces} using Lemma \ref{lem:DeGiorgi}. 
Let us begin with \eqref{eqn:improved_traces2}. We prove the statement only for right traces and only for $\rho$ as the proof is identical in the other cases. Fix $a_n = \eps r_{n}$, $b_n = r_n$, and $C = \eps^{-1} > 1$ so that $b_n \le C a_n$ for all $n$. Also, fix $M$ such that $\|\dot{h}\|_{L^\infty} \le M$. Then, we first claim that for almost every $t\in [0,T]$, $K = [-\frac{1}{2M}, \frac{1}{2M}] \times [1, 2C]$ is compactly contained in $H_t^+$. Indeed, recalling $H_t^+$ is defined by the inequality $y \ge \dot{h}(t) \tau$, we see for $(\tau, y) \in K$, $|\dot{h}(t)\tau| \le \frac{1}{2} < 1 \le y$. 

Now, fix $F_n(t) = \esssup_{a_n < y < b_n} \left[\rho(t, h(t) + y) - \rho^+(t)\right]_+$. Then, from Lemma \ref{lem:DeGiorgi} of the preceding step we know that
\begin{equation}
\lim_{r\rightarrow 0^+} \int_0^T \esssup_{(\tau,y)\in K} \left[\rho(t + r\tau, h(t) + ry) - \rho^+(t)\right]_+ = 0,
\end{equation}
since $K \subset H_t^+$ for almost every $t$.
Let us call the inner quantity $G(r,t)$ so that $\int_0^T G(r,t)\;dt \rightarrow 0$ as $r\rightarrow 0^+$. We will estimate $F_n(t + \tau)$ by $G(c_n,t)$ for a suitable choice of $c_n\rightarrow 0^+$ to conclude.
Note, for $\tau \in [-\frac{a_n}{4M}, \frac{a_n}{4M}]$, we have the pointwise estimate
\begin{equation}\label{eqn:shifted_bound}
\begin{aligned}
F_n(t + \tau) &\le \esssup_{a_n < y < b_n} \left[\rho(t + \tau, h(t + \tau) + y) - \rho^+(t)\right]_+ + \left[\rho^+(t) - \rho^+(t + \tau)\right]_+\\
	&\le \esssup_{a_n - |\tau\dot{h}(t)| < y < b_n + |\tau\dot{h}(t)|} \left[\rho(t + \tau, h(t) + y) - \rho^+(t)\right]_+ + \left[\rho^+(t) - \rho^+(t + \tau)\right]_+\\
	&\le \esssup_{\frac{3}{4}a_n < y < b_n + \frac{1}{4}a_n} \left[\rho(t + \tau, h(t) + y) - \rho^+(t)\right]_+ + \left[\rho^+(t) - \rho^+(t + \tau)\right]_+\\
	&\le G(c_n,t) + \left[\rho^+(t) - \rho^+(t + \tau)\right]_+,
\end{aligned}
\end{equation}
provided $c_n$ is chosen so that
\begin{equation}
\left[-\frac{a_n}{4M},\frac{a_n}{4M}\right] \times \left[\frac{3}{4}a_n, b_n + \frac{1}{4}a_n\right] \subset \left[-\frac{c_n}{2M}, \frac{c_n}{2M}\right] \times \left[c_n, 2Cc_n\right],
\end{equation}
or equivalently, 
\begin{equation}
\max\left(\frac{1}{2}a_n, \frac{b_n + \frac{1}{4}a_n}{2C}\right) \le c_n \le \frac{3}{4}a_n,
\end{equation}
which admits a choice of $c_n$ since $b_n \le Ca_n$. Moreover, because $a_n \rightarrow 0$, $c_n \rightarrow 0$. Finally, averaging both sides of \eqref{eqn:shifted_bound} for $\tau \in [-\delta,\delta]$, for $\delta < \frac{1}{4Ma_n}$ and taking $\delta \rightarrow 0^+$, we see that
\begin{equation}
F_n(t) \le G(c_n,t),
\end{equation}
provided $t$ is a Lebesgue point of each $F_n$ and also $\rho^+$. Finally, integrating in $t$ and taking $n \rightarrow \infty$ yields $\int_0^T F_n(t)\;dt \rightarrow 0$, as desired.

\medskip

Second, let us prove \eqref{eqn:improved_traces1} for $\lambda_1$ via a similar argument. Fix a sequence $a_n \rightarrow 0^+$. Then, take $F_n(t) = \esssup_{-a_n < y < a_n} \left[\min(\lambda_1^+(t),\lambda_1^-(t) - \lambda_1(t, h(t) + y)\right]_+$ and
$$G(r,t) = \esssup_{(\tau,y)\in [-1,1]^2} \left[\min(\lambda_1^+(t),\lambda_1^-(t)) - \lambda_1(t + r\tau, h(t) + ry)\right]_+.$$
Also, let $H(t) = \min(\lambda_1^+(t),\lambda_1^-(t))$. As before, we estimate $F_n(t + \tau)$ pointwise for $\tau \in [-\frac{a_n}{M}, \frac{a_n}{M}]$ as
\begin{equation}\label{eqn:shifted_bound2}
\begin{aligned}
F_n(t + \tau) &\le \esssup_{-a_n < y < a_n} \left[H(t) - \lambda_1(t + \tau, h(t + \tau) + y)\right]_+ + \left[H(t+\tau) - H(t)\right]_+\\
	&\le \esssup_{-a_n - |\tau\dot{h}(t)| < y < a_n + |\tau\dot{h}(t)|} \left[H(t) - \lambda_1(t + \tau, h(t) + y)\right]_+ + \left[H(t + \tau) - H(t)\right]_+\\
	&\le \esssup_{-2a_n < y < 2a_n} \left[H(t) - \lambda_1(t + \tau, h(t) + y)\right]_+ + \left[H(t + \tau) - H(t)\right]_+\\
	&\le G(c_n,t) + \left[H(t + \tau) - H(t)\right]_+,
\end{aligned}
\end{equation}
where now we pick $c_n = \max\left(2a_n, \frac{a_n}{M}\right)$. Averaging \eqref{eqn:shifted_bound2} over $\tau \in [-\delta,\delta]$ for $\delta$ sufficiently small and sending $\delta \rightarrow 0^+$ as before yields
$$F_n(t) \le G(c_n, t),$$
for almost every $t$. Since $c_n \rightarrow 0^+$, we conclude using $G(c_n, t) \rightarrow 0$ pointwise in $t$.


\section{Appendix}\label{sec:appendix}

Here we prove the compactness lemma originally found as \cite[Lemma 1.1]{vasseur_kinetic}:
\begin{proof}[Proof of Lemma \ref{lem:compactness}]
We begin by denoting $\rho_k = \int_{-L}^L f_k \;dv$ and $m_k = \int_{-L}^L f_k\;dv$ so that $\rho_k \rightarrow \rho$ and $m_k \rightarrow m$ in $L^1_{loc}$. Thus, after extracting a subsequence, we have $m_{n_k} \rightarrow m$ and $\rho_{n_k}\rightarrow \rho$ converge pointwise almost everywhere. However, because $f_{n_k}$ is the characteristic function of an interval, the $0$-th and $1$-st moments (in $v$) of $f_{n_k}$ completely determine $f_{n_k}$ as
\begin{equation}
f_{n_k}(t,x,v) = \chi_{\left[\frac{m_{n_k}}{\rho_{n_k}}-\frac{\rho_{n_k}}{2}, \frac{m_{n_k}}{\rho_{n_k}}+\frac{\rho_{n_k}}{2}\right]}(v).
\end{equation}
Therefore, the pointwise convergence is sufficient to conclude
\begin{equation}
f_{n_k} \rightarrow \chi_{\left[\frac{m}{\rho}-\frac{\rho}{2}, \frac{m}{\rho}+\frac{\rho}{2}\right]}(v) := g(t,x,v).
\end{equation}
By Lebesgue dominated convergence, we conclude $f_{n_k} \rightarrow g$ in $L^1_{loc}$. Now, since for any subsequence of $f_n$, there is a further subsequence which converges to $g$ in $L^1_{loc}$, we conclude that the entire family $\{f_n\}$ converges to $g$ in $L^1_{loc}$. Finally, because we also have $f_n \weakstar f$ in $L^\infty$, we conclude $f = g$ and we have demonstrated explicitly that $f$ is a characteristic function of an interval (in $v$).
\end{proof}

Here we provide a proof to \ref{cor:Dafermos}. In fact, the result applies to a general class of fluxes provided that the solutions considered satisfy the strong trace property of Definition \ref{defn:strong_traces}. The result in the case of $BV$ solutions is originally due to Dafermos, but we refer to \cite{serre_vasseur} for a proof. This proof follows that of \cite[Lemma 6]{leger_vasseur}:
\begin{proof}[Proof of Corollary \ref{cor:Dafermos}]
Fix $u$ a solution to \eqref{eqn:Euler}, $f$ the corresponding flux of \eqref{eqn:Euler}, and $h:[0,T] \rightarrow \R$ a Lipschitz curve. Pick $\psi(x) \in C^\infty$ with $0 \le \psi \le 1$ and $\supp(\psi)\subset (0,1)$ and let $\psi_{\eps}(x)$ be the standard mollifier, $\psi_{\eps}(x) = \eps^{-1}\psi\left(\frac{x}{\eps}\right)$. Next, we pick $\Psi_{\eps}(x) = \int_x^\infty \psi_{\eps}(y) - \psi_{\eps}(-y)$. Then, $\Psi_{\eps}$ is smooth, compactly supported and by Fubini, converges to $0$ in $L^1(\R)$. Moreover, for any $\varphi \in C^\infty_c([0,T])$ and any $v:\R^+\times \R \rightarrow \R^m$ strong trace in the sense of Definition \ref{defn:strong_traces},
\begin{equation}\label{eqn:Dafermos1}
\begin{aligned}
\int_0^\infty\int_{\R} \varphi(t)\Psi_{\eps}^\prime(x - h(t))v(t,x) \ dx dt &= \int_0^\infty\int_{\R} \varphi(t)\psi_{\eps}(x) \left[v(t,h(t) + x) - v(t,h(t) - x)\right] \ dxdt\\
	&= \int_0^\infty\int_{\R} \varphi(t)\psi(x) \left[v(t,h(t) + \eps x) - v(t,h(t) - \eps x)\right] \ dxdt.
\end{aligned}
\end{equation}
Finally, the right hand side converges to $\int_0^\infty \left[v^+(t) - v^-(t)\right]\varphi(t) \ dt$ by the strong trace property. Therefore, testing the weak form of \eqref{eqn:Euler} with $\varphi(t)\Psi(x - h(t))$, to obtain
\begin{equation}
\begin{aligned}
0 &=-\int_0^\infty\int_{\R} \partial_t \varphi(t)\Psi_{\eps}(x-h(t))u(t,x) \;dxdt\\
	&\qquad+ \int_0^\infty\int_{\R}\varphi(t)\Psi_{\eps}^\prime(x-h(t))\dot{h}(t) \;dxdt\\
	&\qquad- \int_0^\infty\int_{\R}  \varphi(t)\Psi_{\eps}^\prime(x-h(t))f(u(t,x)) \;dxdt\\
	&:= I_1 + I_2 + I_3.
\end{aligned}
\end{equation}
We note first that because $\Psi_{\eps} \rightarrow 0$ in $L^1$, $I_1 \rightarrow 0$. Second, because $h$ is Lipschitz, $\dot{h}(t)u(t,x)$ is strong trace in the sense of Definition \ref{defn:strong_traces} with traces $\dot{h}(t)u^\pm(t)$. Therefore, \eqref{eqn:Dafermos1} implies
\begin{equation}
\lim_{\eps\rightarrow 0^+} I_2 = \int_0^T \varphi(t)\dot{h}(t)\left[u^+(t) - u^-(t)\right] \;dt.
\end{equation}
Third, because $f$ is continuous, $f(u)$ is strong trace in the sense of Definition \ref{defn:strong_traces} by a simple compactness argument using the Lebesgue dominated convergence theorem. Then, \eqref{eqn:Dafermos1} implies
\begin{equation}
\lim_{\eps\rightarrow 0^+} I_3 = \int_0^T \varphi(t)\left[f(u^+(t)) - f(u^-(t))\right] \;dt.
\end{equation}
Using the fundamental lemma of calculus of variations, we conclude the Rankine-Hugoniot jump condition,
\begin{equation}
f(u^+(t)) - f(u^-(t)) = \dot{h}(t)\left[u^+(t) - u^-(t)\right] \qquad \text{for almost every }t\in [0,T].
\end{equation}
A similar argument for a continuous entropy, entropy-flux pair $(\eta,q)$, yields the entropy inequality
\begin{equation}
q(u^+(t)) - q(u^-(t)) \le  \dot{h}(t)\left[\eta(u^+(t)) - \eta(u^-(t))\right] \qquad \text{for almost every }t\in [0,T].
\end{equation}
We therefore conclude for almost every $t$, either $u^+(t) = u^-(t)$ or $(u^-(t),u^+(t),\dot{h}(t))$ is an entropic shock. 
\end{proof}

\bibliographystyle{amsplain}
\bibliography{Full_Draft}

\end{document}